\documentclass{article}
\usepackage{amssymb}
\usepackage{amsmath}
\usepackage{amscd}
\usepackage{amsfonts}
\numberwithin{equation}{section}
\usepackage{html}
\RequirePackage{amsthm}
\RequirePackage{amssymb}

\def\ra{\rightarrow}

\def\longra{\longrightarrow}

\def\mpo{\mapsto}

\def\bsh{\backslash}

\def\ify{\infty}

\def\ov{\overline}
\def\sbs{\subset}

\def\ts{\times}

\def\a{\alpha}
\def\b{\beta}
\def\d{\delta}
\def\D{\Delta}
\def\e{\varepsilon}
\def\g{\gamma}
\def\G{\Gamma}

\def\lb{\lambda}

\def\om{\omega}
\def\Om{\Omega}

\def\Si{\Sigma}
\def\ve{\varepsilon}
\def\vp{\varphi}
\def\z{\zeta}

\def\cB{{\cal B}}
\def\cC{{\cal C}}
\def\cD{{\cal D}}

\def\cO{{\cal O}}

\def\cS{{\cal S}}

\def\cW{{\cal W}}
\def\cX{{\cal X}}

\def\tcx{\tilde{\cal X}}

\def\tb{\tilde{B}}

\def\tz{\tilde{Z}}

\def\dz{\frac{\partial}{\partial z}}

\def\ba{\mathbb A}
\def\bc{\mathbb C}

\def\bg{\mathbb G}

\def\bn{\mathbb N}
\def\bp{\mathbb P}
\def\bq{\mathbb Q}
\def\br{\mathbb R}

\def\bz{\mathbb Z}

\def\dm{d\mu}
\def\dn{d\nu}

\def\pa{\partial}

\def\un{{\rm 1\mkern-4mu I}}

\def\build#1_#2^#3{\mathrel{
\mathop{\kern 0pt#1}\limits_{#2}^{#3}}}

\def\dc{d^{c}}
\def\ddc{d\dc}

\def\nint{{\int\!\!\!\!\!\!\nabla}}

\def\Kob{\mathrm{Kob}}

\def\Kob{\mathrm{Kob}}

\newtheorem{thm}{Theorem}[subsection]
\newtheorem{defn}[thm]{Definition}

\newtheorem{warning}[thm]{Warning}

\newtheorem{claim}[thm]{Claim}

\newtheorem{sch}[thm]{Scholion}
\newtheorem{inter}[thm]{Intermission}

\newtheorem{fact}[thm]{Fact}
\newtheorem{factd}[thm]{Fact/Defintion}
\newtheorem{choiced}[thm]{Defintion/Choice}
\newtheorem{summaryd}[thm]{Summary/Defintion}
\newtheorem{cor}[thm]{Corollary}
\newtheorem{rmk}[thm]{Remark}
\newtheorem{prop}[thm]{Proposition}
\newtheorem{lem}[thm]{Lemma}

\newtheorem{ex}[thm]{Example}
\newtheorem{exd}[thm]{Example/Definition}

\title{The Bloch Principle}
\author{Michael McQuillan}
\date{}
\begin{document}
\maketitle
\noindent{\scriptsize We formulate and prove an optimal version
for quasi-projective surfaces
of A.~Bloch's dictum,
``Nihil est in infinito quod prius non fuerit in
finito'' by way of a complement
to a theorem of J.~Duval.}

\subsection*{Introduction}

In \cite{B}, A.~Bloch, obtained a substantive
generalisation of Montel's theorem (maps
of the disc to 
$\bp^1 \bsh \{ 0,1,\infty \}$ are normal) to the case of $\bp^2 \bsh 
\{\hbox{4 lines}\}$, with H.~Cartan's thesis, \cite{cart}, being
the same in arbitrary dimension, {\it i.e.}
$\bp^n \bsh \{n+1\, \hbox{planes}\}$, $n\geq 3$, or, perhaps,
$n\geq 2$ given a missing proof of a lemma
in the former which Cartan provided. The precise
statement is as follows: choose a basis $X_i$ 
of a $n+1$ dimensional vector spaces, and thus
identify the simple normal crossing divisor, $B$,
consisting of $n+2$ planes in $\bp^n$ in general 
position with the coordinate hyperplanes
$X_i=0$ together with the hyperplane 
$X_0 + \cdots + X_{n+1} = 0$, similarly for
every $ I \subset \{ 0 , \ldots , n+1 \}\,\,$,
$2\leq |I|\leq n\,\,$ there are diagonal
hyperplanes $\Lambda_I= \build\sum_{i \, \in \, I}^{} X_i = 0$,
every $\Lambda_I\bsh B$ is covered by $\bg_{\mathrm{m}}$'s,
and the union $Z$ of the $\Lambda_I$ is,
an identification
of the smallest Zariski closed subset $Z\sbs \bp^n$, such
that (Theorem of Borel) every non-trivial (a precision
which will, henceforth be eschewed) entire map $f:\bc\ra\bp^n\bsh B$ factors
through $Z$, then a sequence of holomorphic
discs $f_n:\D\ra \bp^n\bsh B$ which is not
(compact-open sense) arbitrarily close to $Z$
admits a subsequence converging uniformly
on compact sets.  

Bloch, \cite{B}, makes his pleasure in the
theorem quite clear, and, arguably with good
reason, since it is way more difficult than
the relatively trivial assertion and/or exercise
in the definition of the Wronskian that 
every  entire map $f:\bc\ra\bp^n\bsh B$ factors
through $Z$. The latter is, of course, for $n=2$, the
motivating ``infinito'' to which his dictum referred,
while the uniform convergence for discs
not arbitrarily close to $Z$ is the ``in finito''.
Nevertheless, for
a very long time this theorem of 
Bloch/Cartan's thesis was wholly sui generis 
even to the point of there being no other
non-obvious ({\it i.e.} not negatively
curved or similar) examples of quasi-projective
varieties $(X,B)$ (so, inter alia, projective 
varieties $X$ if the boundary $B$ is empty),
such that one could assert that the property,

\noindent{\bf Infinito} {\it There is a proper Zariski closed
subset $Z\sbs X$ (implicitly without generic points
in $B$) such that any  entire map 
$f:\bc\ra X\bsh B$ factors through $Z$.}

\noindent was occasioned by a stronger theorem ``in finito'',
such as normal convergence for sequences of discs not
arbitrarily close to $Z$. To some extent this
changed with the appearance of Brody's lemma,
\cite{brody}, {\it i.e.} for $X$ not just
projective, but compact, infinito holds
with $B=Z=\emptyset$ iff every sequence
of holomorphic discs $f_n:\D\ra X$ has a
uniformly convergent sub-sequence, together
with Mark Green's complement, \cite{green}, in the 
presence of a boundary, {\it i.e.} everything as above, so again 
$Z=\emptyset$, and furthermore if $B=\sum_i B_i$ is a sum
of Cartier divisors then for every set of components
$I$, no stratum $B_I:=\cap_{i\in I}B_i\bsh \cup_{j\notin I} B_j$
admits an entire curve. Consequently although Brody's lemma,
and modulo the boundary restrictions, Green's complement,
turned Bloch's principle into a theorem for $Z=\emptyset$ 
the Bloch-Cartan theorem remained sui generis as 
the only example of the Bloch principle with a non-empty exceptional
set $Z$.

The first step to new examples occurred when
infinito was established, \cite{ihes}, for
surfaces, $X$, of general type with 
$\mathrm{s}_2=\mathrm{c}_1^2 - \mathrm{c}_2>0$. However,
under such hypothesis, unlike the Bloch theorem, 
$Z\bsh B$, may admit,
projective, $\bp^1$, or, affine lines,
$\ba^1$, and should this occur there are sequences of discs
$f_n:\D\ra X\bsh B$ which are not arbitrarily
close to $Z$ but which do not converge uniformly
on compact sets. Such discs are, however, rather
small, and, appropriately understood do actually
converge, {\it i.e.} not uniformly, but in the
sense of Gromov, to a disc with  bubbles. Indeed,
many so called ``counterexamples'' to the 
Bloch principle are of this type, which leads us
to pose the best possible variant,

\noindent{\bf In finito} {\it A sequence of holomorphic discs $f_n:\D\ra X\bsh B$
not arbitrarily close to $Z\cup B$ (or possibly just $Z$) has
a subsequence converging to a disc with bubbles.}

\noindent so that, conjecturally, \cite{icm}, the Bloch principle
becomes infinito iff in finito, and, this was established
for the said class of surfaces in \cite{bloch}, which,
in turn, apart from the Bloch-Cartan theorem, constituted
not only the unique examples of the Bloch principle with
$Z\neq\emptyset$, but had a feature that Bloch-Cartan does
not, {\it viz:} the necessary appearance of bubbles in finito.

Of course, \cite{bloch}, gave a proof of the Bloch
theorem for $\bp^2\bsh\{\hbox{4 lines}\}$ en passant,
and, although cleaner, neither here, nor in the
study of quasi-projective surfaces $X\bsh B$
with $\mathrm{s}_2>0$ did it liberate itself from the
same methodological flaw of Bloch-Cartan, {\it i.e.}
rather than attempt to prove infinito implies in finito
directly, one was simply doing the infinito proof
for the corresponding class of varieties better, 
and, as it happens, with a similar, albeit
reversed, inductive structure to Bloch-Cartan, with
a view to proving the in finito assertion directly. Worse, albeit
this is properly discussed below, such a
variant of the Bloch-Cartan strategy can only work
in the presence of optimal theorems on the variation in 
the normal direction of discs which are approximately 
solutions of certain O.D.E.'s. This holds for 
$\mathrm{s}_2>0$, \cite{marco}, but may fail otherwise,
\cite{ext}. 

As such, the methodology was clearly un-sustainable,
and the key development permitting one to go directly
ex infinito in finitum is a much better version of
Brody's lemma due to Julien Duval, \cite{duval},
to wit, for $X$ compact, let $f_n:\D\ra X$ be a 
sequence of discs defined up to the boundary
such that the ratio of the length of the same
to the area goes to zero, then, for $A$ the
(more correctly a up to subsequencing) resulting
closed positive (Ahlfors) current: if $A$ has mass on a
compact set $K\sbs X$, then there is an entire map
with bubbles cutting $K$. This is a much more 
finely tuned instrument than \cite{brody}, {\it e.g.},
already for $Z=B=\emptyset$,
it gives not just infinito iff in finito but, \cite{duval}, iff
Gromov's isoperimetric inequality. Another corollary,
was that \cite{bloch} could supersede itself,
{\it i.e.} the ideas therein, and, in fact,
the O.D.E. free ones, could be used to establish
the Bloch principle for any quasi-projective surface.
This, modulo bearing in mind that op. cit. is a
pre-print, so, we repeat some of its contents 
to ensure no logical appeal to it, is the
content of the present, which we may now summarise.
 
In the first place, contrary to the assertion in
the introduction of \cite{duval}, it is not true
that McQuillan theory uses closed currents in the
sense of Ahlfors, \ref{eq:cint}, but rather in the
sense of Nevanlinna, \ref{eq:cnev}. Alternatively
it chooses a point, \S \ref{SS:area}, on a Riemann-surface
should this have a border. This choice is, ultimately
just a means to an end, {\it i.e.} establishing choice
free theorems, but it is essential in bringing
algebraic geometry into play by way of \ref{ex:compact},
so that in contradistinction to ``usual Nevanlinna
theory'' the term at the origin goes from the least 
to the most
important. This also poses problems, the first
of which is to extract closed positive currents
in the Nevanlinna sense from discs, which is
sufficiently more difficult than the Ahlfors
sense, that we've only done it, \ref{prop:closed},
under the hypothesis of projectivity of the target.
Similarly we need Duval's theorem in the Nevanlinna
rather than the Ahlfors sense, albeit we reduce
the former to the latter in \ref{fact:almost} and
\ref{fact:nobloch}. Such considerations are enough
to get us to the starting line if 
$B=\emptyset$, {\it i.e.} if infinito holds but
in finito fails there is a closed, and without
loss of generality nef., positive current
$T$ arising from discs, in the Nevanlinna sense, supported on $Z$.
In the case of a non-empty boundary, however,
one must give conditions that bubbles do not
form in the boundary, \ref{defn:nobubbles},
which in the first approximation amounts to
Green's condition on strata above but for
$\ba^1$'s rather than $\bc$'s, \ref{lem:green2}.
Such a condition holds in an almost \'etale neighbourhood
of the boundary of the minimal model of any quasi-projective
surface which is not covered by $\ba^1$'s, but
such an almost (\'etale) boundary condition does not
imply no-bubbling in the boundary (there are counterexamples)
whereas the condition of being a minimal model
does, \ref{prop:nobubbles}, which, again, 
since it's a bit subtle, we
only prove in the quasi-projective case, even
though it has a perfect almost complex sense.
This suggests, and, it's necessary, that Bloch's principle 
be formulated on the minimal, or better canonical
model. 

At this juncture there are some subtleties,
which, in turn, are complicated by the currently
overwhelming tendency in the study of complex
hyperbolicity to treat the canonical
model, if at all, in qua, 
rather than ex qua. For example,
the Green-Griffiths
conjecture, {\it i.e.}  infinito holds iff general type, 
is posed for entire maps to surfaces
of general type, rather than the more conformally
natural hypothesis of entire maps to the Vistoli
covering champ (orbifold), \ref{factd:vistoli}, of their canonical model,
which, after all, is where its K\"ahler-Einstein
metric is defined. As such, although the Bloch
principle has been formulated in \ref{summaryd:crit},
and proved in \ref{cor:bloch}, in a way that
wholly responds to this
in qua tendency, it is most easily summarised
ex qua. More precisely,
if a pair $(S,B)$ with log canonical singularities,
and $K_S+B$ ample is given, then the singularities
of $S$ are either quotient or elliptic Gorenstein.
The latter, $P$, say, are what one finds at cusps
of ball or bi-disc quotients, and in a neighbourhood
of $V\supset P$, the K\"ahler-Einstein and Kobayashi
metrics on $V\bsh P$ are complete and mutually
comparable. 
As such, notation not withstanding,
these should always be treated as belonging to
the boundary, which might be better written
$B\cup P$. Amongst the quotient singularities, $Q$,
however, those in the interior
$S\bsh B$ do not naturally (ex qua), cf.
\ref{summaryd:crit}, belong to the boundary,
so we replace $S$ by a champ de Deligne-Mumford,
$\cS\ra S$ smooth over $Q$, and, otherwise
isomorphic to $S$ on which we have,

\noindent{\bf Main Theorem} (\ref{cor:bloch}) {\it Notations as
above, then the following are equivalent,

\smallskip

\noindent {\em (Infinito)} There is a proper Zariski closed
$Z\sbs \cS$ such that every entire map $f:\bc\ra \cS\bsh \{B\cup P\}$
factors through $Z$.

\smallskip

\noindent {\em (In finito)} A sequence of holomorphic discs
$f_n:\D\ra \cS\bsh\{B\cup P\}$ 
(equivalently, \cite{kobcodim2}, $f_n:\D\ra S\bsh\{B\cup P\cup Q\}$)
which is not arbitrarily
close (compact open sense) to $Z$ has a sub-sequence
converging to a disc with bubbles in $\cS$ and, 
\ref{defn:nobubbles}, without bubbles in the boundary.
}
\smallskip

Which, if one likes things in qua, 
{\it i.e.} one starts, \ref{summaryd:crit}, from a smooth surface
$\tilde{S}$ with simple normal crossing
boundary $\tilde{B}$ of which $(S,B)$ is
the canonical model, then the appropriate
statement, \ref{cor:bloch}, is obtained on replacing the
champ $\cS$ by another, $\cS'$,
according to the rule: $\cS'$ is isomorphic
to $\cS$ around $q\in Q\cap (S\bsh B)$ if 
the pre-image of $q$ is a connected component
of $\tilde{B}$, and isomorphic to $S$ around
$q$ otherwise.
Irrespectively, one should be aware that the above is
not simply theoretical, {\it i.e.} while as noted,
\cite{bloch}, provides another (more difficult) proof of in finito whenever
$\mathrm{s}_2(X,B)>0$, there is no similar proof
otherwise, but there is a proof, \cite{web}, of
infinito under hypothesis such as $13\mathrm{c}_1^2>9c_2$.

To some extent, the above theorem constitutes
the limit of our present 
ambition, but one should always have an
eye towards trying to understand the
geometry of algebraic surfaces as well
as we understand Riemann surfaces or
complete 3-manifolds. Ultimately, one might
hope, \cite{icm}, that this involves some
sort of uniformisation theorem defined via
extremal discs. In the first instance, however,
it involves Kobayashi's
intrinsic metric, which by the above, in the
presence of infinito, is a
complete metric off $Z$, so that, around 
this exceptional set of rational and elliptic
objects, we should aim to quantify its degeneration. 
Such quantification might, therefore,
be considered
a crypto Bloch principle. Inevitably, it
must involve a resolution of singularities
of $(S,Z+B)$, so we defer to
\ref{prop:degen} for the exact statement.
The key feature, however, is what occurs
at a generic point of $Z$, where,
by way of notation, we take $x$ to be a normal
coordinate to $Z$, $y$ longitudinal, and find for every
$\d>0$ a constant $c(\d)>0$, such that a 
lower bound for the Kobayashi
metric close to $Z$ is given by,
$$
\dfrac{c(\d)}{\log^{4} \vert x\vert (\log\vert\log\vert x\vert\vert)^{1+\d}}
\cdot \bigl( dx\otimes d\bar{x} + \vert x\vert^2 dy\otimes d\bar{y}\bigr)$$
which is very close to best possible, but at least
in the normal direction appears to be slightly out by the
scaling factor $(\log\vert x\vert)^{-4}(\log\vert\log\vert x\vert\vert)^{-1-\d}$.

Finally, 
let us say something about the proof.
To the preliminaries already encountered,
is to be
added, in all dimensions, some tautologies, \ref{SS:taut}, about
what happens on taking (logarithmic) derivatives
of  discs in say, $\bp(\Om^1_{X}(\log D)$,
for $D$ a simple normal crossing divisor.
Roughly speaking, this amounts to finding a derived current $T'$
pushing forward to a current $T$
defined by discs on $X$, whose intersection with
the tautological bundle $L$ 
on $\bp(\Om^1_{X}(\log D)$
satisfies $L\cdot T'\leq D\cdot T$,
albeit there are a series of much more precise
statements \ref{fact:taut1}-\ref{fact:taut3},
involving pointed Riemann surfaces,
a complete metric on $X\bsh D$, the
multiplicity free intersection with $D$,
and the distance of the point from $D$. 
Of all the preliminary sections \S 1.1- \S 2.1,
the surprise for the necessity of which
at the beginning of the 21st century notwithstanding,
probably only this tautological one, \S 1.4, can be
considered remotely definitive- other
key  preliminary propositions \ref{prop:closed},
and \ref{prop:nobubbles}
having recourse to the crutch of projectivity.
This said, let us come to the point, once
the (complex) dimension exceeds 1 there is,
contrary to some false and bogus conjectures in circulation,
absolutely no way to get from the tautological
upper bound on the tautological degree implied by 
our closed formula, \ref{fact:taut2}, and 
the length area principle to an identical  bound
for the (log) canonical degree. The sense here, however, of
``dimension 1'' should be very broadly interpreted. 
For example in the case of the Bloch-Cartan
theorem it means discs limiting on the space of
lines in $\bp^n$ viewed as a $n$th order
ODE, or, equivalently sub-variety of the
$n$th jet space, or solutions of a first order
O.D.E.  in the case of
surfaces with $\mathrm{s}_2>0$. In situations such
as these where the discs limit on a 
sufficiently nice O.D.E. 
one can bound the canonical by the
tautological degree, or, more or
less equivalently, the sectional
curvature (understood at the weaker
Nevanlinna level) is at most the
Ricci curvature along the solutions
of the O.D.E., albeit the arguments,
be it \cite{B}, \cite{cart}, or
\cite{bloch} can get rather
involved. 
Thanks to
Duval's theorem, however, and \ref{prop:nobubbles},
the only O.D.E.'s obstructing
ex infinito in finitum have order zero, and
are better referred to as algebraic curves,
and we have,

\noindent{\bf Main Lemma} (\ref{prop:main}) {\it Let $T$ be a closed
positive current on an algebraic surface, $X$,
afforded by (pointed) 
discs in the Nevanlinna sense, \ref{eq:cnev},
such that,

(a) $T$ is supported in a simple normal
crossing divisor $D$.

(b)The origins of the discs do  not
accumulate on $D$.

\noindent then for $T'$ on $\bp(\Om^1_X(\log D))$ the current
associated to the logarithmic derivatives of the discs,
and $L$ the tautological bundle on the same,
$$L\cdot T'\geq (K_X+D)\cdot T$$
}
With the content of the proof being
the interplay between the hypothesis
(b) and adjunction. The above
main theorem/
Bloch principle, and its in qua variants,  
appropriately,
and necessarily (to repeat: in all cases
the principle requires a minimum
of minimality of the model) 
corrected for quotient and elliptic
Gorenstein singularities is, \ref{cor:bloch},  an immediate
corollary of the lemma. The 
lemma
has, however, further utility, and
is, for example, the reduction that the programme, \cite{web},
for resolving the Green-Griffiths
conjecture 
aims
for. 

Thus, while the regrettably long preliminaries of \S 1 are
valid in all dimensions, their application
to the Bloch principle in \S 2 is 
a surface theorem,
which, inter alia, and much after the event, 
permits, via \cite{duval},
Bloch's theorem
to be
considered trivial. One is no closer, however,
to a similar statement in all
dimensions. In
the quasi-projective setting there are, in all
probability,  examples where no model
affords a no bubbling in the boundary lemma
\`a la \ref{prop:nobubbles}, while even in
the projective case, large discs defining
a current supported in a
proper sub-variety isn't
much information when this has dimension more than 1.  
As such, Cartan's thesis 
has a particular 
inductive (in the dimension) structure
which is not
true in higher dimension, and which causes  
it to remain sui generis.

\newpage

\section{Nevanlinna Theory}\label{S:nev}

\subsection{Area and double integration}\label{SS:area}

Throughout this section $(\Sigma, p)$ will be a pointed
Riemann surface with boundary. The boundary will always
be understood to be regular, {\it i.e.} it admits
a Green's function, or better, for the present, a psh. function 
$g:\Sigma\ra [-\infty, 0]$ vanishing on the boundary,
and finite on $\Sigma\bsh \{p\}$ satisfying:
\begin{equation}\label{eq:green}
\ddc g=\d_p
\end{equation}
and we denote by $z$ a function around $p$ such that,
\begin{equation}\label{eq:z}
g=\log\vert z\vert^2 + O(\vert z\vert)
\end{equation}
For $r\in (-\infty, 0]$, let $\Si_r$ be the open set of points
where $g<r$. This is relatively compact for $r<0$, and we may
integrate, say:
\begin{equation}\label{eq:int}
\int_{\Si_r}:A^{1,1}(\Si)\ra\bc:\om \mpo \int_{\Si_r}\om
\end{equation}
The choice of the point $p$ permits Nevanlinna's
variant on this,
\begin{equation}\label{eq:nev}
\nint_{\Si_r}:A^{1,1}(\Si)\ra\bc:\om \mpo \frac{1}{2}\cdot \int_{-\infty}^r dt \int_{\Si_t}\om
\end{equation}
And of course we extend these definitions whether
to $r=0$, or $\om$ just a signed measure, or both,
provided they continue to have sense as absolute
integrals. 

The exact relation of Nevanlinna's definition with
geometry is unclear- after all it involves a choice.
It is, therefore, perhaps best to treat it as a
technical tool whose utility comes from:

\begin{fact}\label{fact:cart} {\em Let $\bar D$ be a metricised Cartier divisor on $\Si$, with 
$\un_D \in \G ({\mathcal O}_{\Si} (D))$ the tautological section, then,

\begin{eqnarray*}
{\int\!\!\!\!\!\!\nabla}_{\Si_r} 
\mathrm{c}_1 (\bar D) & :=  &\sum_{-\infty < g(z)  < r} \mathrm{ord}_z (D
) \cdot \frac{(r-g(z))}{2}  + \mathrm{ord}_0 (D) \frac{r}{2}\\ 
&& - \int_{\partial \Si_r} \log \Vert \un_D \Vert \cdot \dc g + \lim_{z 
\ra p} \log \dfrac{\Vert \un_D \Vert}{\,\,\vert z \vert^{\mathrm{ord}_0 (D)}} 
\end{eqnarray*}
}
\end{fact}

\begin{proof} 
We're always working in a complex setting, so
by definition: $\mathrm{c}_1 (\bar D)=-\ddc \log \Vert \un_D \Vert^2$,
almost everywhere, after which we have an exercise in integrating by parts.
\end{proof}

This leads to the principle application of
\ref{eq:nev}, which doesn't have a parallel
under \ref{eq:int}, to wit:

\begin{ex}\label{ex:compact}
{\em Let $X$ be a compact complex space, or indeed 
an analytic champ de Deligne-Mumford, and
$\overline D$ a metricised divisor on it,
then for any map $f:\Si\ra X$ from any
(pointed) Riemann-surface such that $f(p)\notin S$, and any $r$,
$$
\nint_{\Si_r}
f^*\mathrm{c}_1 (\bar D) \geq 
 \log \Vert \un_D \Vert(p)
- \sup_X \log {\Vert \un_D \Vert} 
$$
The supremum in question 
depends only on $X$ and $D$, we may as well normalise
so that it's zero, so the only obstruction to
a positive intersection in the Nevanlinna sense
of $f$ with $D$ is the distance of $f(p)$ to $D$.
}
\end{ex}
Here is another example of the utility of \ref{fact:cart},

\begin{exd}\label{exd:euler}
{\em Suppose that $\Si$ is hyperbolic, and let
$\om$ be the Poincar\'e metric. The form $\pa g$
is meromorphic, and the (signed) sum of its zeroes
and poles in any $\Si_r$ is the Euler characteristic
$\chi(r)$ of the same- just consider the compact
double obtained by Schwarz reflection in the boundary.
The Nevanlinna variant of the Euler caharacteristic is given by,
$$
E_\Si(r):=\frac{1}{2}\cdot \int_{-\infty}^0 \chi(\Si_t\bsh p) \, dt \,+\, \frac{r}{2}
$$
and the definition of constant curvature -1 implies,
for $\pa$ a field dual to $\pa g$, and $z$ as in \ref{eq:z},
$$
\nint_{\Si_r} \om + E_\Si(r) 
+ \log\Vert\dz\Vert_\om(p) =
\int_{\pa\Si_r} \log \Vert\pa\Vert_\om\, 
\dc g \leq  \log\bigl[(\mathrm{length}_{\om})(\pa\Si_r)
\bigr]
$$ 
which, although there are more interesting length-area
isoperimetric inequalities using \ref{eq:int} rather
than \ref{eq:nev}, their proofs are much more difficult.
}
\end{exd}

\subsection{Closedness and positivity}\label{SS:closed}

Let us consider in more detail the situation presented
in \ref{ex:compact}, with $\om$ a positive
$(1,1)$ form on $X$, and $f_n:\Si_n\ra X$ maps
from (pointed) Riemann surfaces. The alternatives
\ref{eq:int} \& \ref{eq:nev} as to how we
integrate over $\Si_n$ lead to competing definitions
for bounded currents on $X$, {\it viz:}
\begin{equation}\label{eq:cint}
A_n (r) : A^{1,1} (X) \ra {\mathbb C} : \tau \mpo 
\left( \int_{\Si_{n,r}} \, f_n^* \, \om \right)^{-1} {\int}_{\Si_{n,r}} \, f_n^* \, \tau \,
\end{equation}
and its Nevanlinna counterpart:
\begin{equation}\label{eq:cnev}
T_n (r) : A^{1,1} (X) \ra {\mathbb C} : \tau \mpo 
\left( \nint_{\Si_{n,r}} \, f_n^* \, \om \right)^{-1} {\nint}_{\Si_{n,r}} \, f_n^* \, \tau \,
\end{equation}
Where in either case we will implicitly suppose that
there is $r\in (-\infty,0)$ such that as $n\ra\infty$:
\begin{equation}\label{eq:unbounded}
 \int_{\Si_{n,r}} \, f_n^* \, \om\, \ra \infty, \,\,\,
\mathrm{respectively,}\,\,\,
 \nint_{\Si_{n,r}} \, f_n^* \, \om\, \ra \infty
\end{equation}
and one should bear in mind that we are not asserting
any a priori relation between the respective conditions.
Under the former of these conditions, and up to moving
$r$ a little, the length area principle implies, without
much work, that a subsequence of the $A_n(r)$ converges
to a closed positive current. The corresponding statement
for $T_n(r)$ is more difficult, and we'll limit our
discussion to the case where all the $f_n$ are maps
from the unit disc, $\D$, pointed in the origin. 
In such a situation it is convenient to change the
notation according to,
\begin{warning}\label{warning:notation}
{\em Should we be discussing uniquely sequences of discs
then $\D (r)$ will be the (pointed) disc of radius $r$.
The solution of \ref{eq:green} is, of course,
$\log\vert z\vert^2$, so for $\tau$ $(1,1)$ on the disc: 
\begin{equation}\label{eq:notation}
\nint_{\D_{r}} \, \tau\, =\, \int_0^r\, \frac{dt}{t}\int_{\D (t)}\, \tau
\end{equation}
}
\end{warning}
We will further limit our attention to $X$ projective,
so in the first instance $\bp^n$ for some $n$, and we
make,
\begin{choiced}\label{defn:fubini} A Fubini-Study metric
on $\bp^n$ depends on a choice of basis. Make such a
choice and take the $\om$ occurring whether in \ref{eq:cint}
or \ref{eq:cnev} to be the resulting Fubini-Study form.
\end{choiced}
Which we will suppose in,
\begin{lem}\label{lem:origin} Let $x\in\bp^n$, then there
is an open (archimedean) neighbourhood $U\ni x$, and
constants $c_1$, $c_2$, and $N>n$, depending only on $U$ and the
choice of  \ref{defn:fubini} such that for every
$f : \D \ra \bp^n$, and $0 < r < 1$. 
$$
R\Vert df(z) \Vert_\om \leq c_1
{\int\!\!\!\!\!\!\nabla}_{\D (R)} \, f^* \, \om
+ c_2
$$
for $2 |z| < R$, where $R = \mathrm{min} \{ \frac{r}{N} \mathrm{exp} (-
{\int\!\!\!\!\!\!\nabla}_{\D (r)} \, f^* \, \om) 
, r \}$, and the implied norm on the disc is just
the Euclidean one.
\end{lem}
\begin{proof}  Everything is homogeneous under the unitary group so we may as well 
say $x = [1,\ldots , 1]$, in the basis $X_i$ generating $\cO_{\bp^n} (1)$
and affording $\om$.  A disc is Stein, so fixing a trivialisation of $f^* 
\cO_{\bp^n} (1)$ allows us to write $f = [f_0 , \ldots , f_n]$ in these
coordinates for $f_i$ functions on the disc without common factor.
As such if $H_i$ is the $i^{\rm 
th}$-coordinate hyperplane, we may suppose $U\cap H_i=\emptyset$, $\forall i$, so:
$$
{\int\!\!\!\!\!\!\nabla}_{\D (r_0)}  f^*  \omega =
 - \int_{| z | = r} \log 
\frac{\Vert f^* X_i \Vert}{ \Vert f^* X_i \Vert  (0)} \frac{d\theta}{2\pi}+ 
\sum_{0 < 
\vert z \vert < r} \hbox{ord}_z (f^* H_i) \log \frac{r}{\vert z \vert} 
$$

Now for convenience put, $T_f (r)$ to be the left hand side of the
above, and take $U$ sufficiently small so that the infimum over
$u\in U$ and $i$ of $\Vert X_i\Vert(u)$ is $N^{-1}> 0$. Whence 
if $f(z) \in H_i$ with $\vert z \vert < r$, then:
$$
\vert z \vert \geq \frac{r}{N} \exp (-T_f (r)) \, .
$$
So put $R = \mathrm{min} \left\{ r , \frac{r}{N} \exp (-T_f (r)) \right\}$ then for 
$\vert z \vert < R$, $f_i (z) \neq 0$. In particular if we identify affine 
$n$-space $\ba^n$ with $\bp ^n \setminus  H_0$ then $\D (R)$ maps to $\ba^n$ under $f$ 
with standard coordinates $g_i = f_i / f_0$ and each $g_i$ is a unit on $\D (R)$.

Consequently Jensen's formula gives,
$$
\frac{g'_i (z)}{g_i (z)} = \int_0^{2\pi} \log \vert g_i  (Re^{i\theta}) 
\vert \, \frac{2 
Re^{i\theta}}{(Re^{i\theta} - z)^2} \frac{d\theta}{2\pi} 
$$
We require to control the modulus of the logarithm in the integrand, so to this 
end observe that inversion on $\bg _m^n = \{ x_1 \ldots x_n \neq 0 \} \subset 
\ba^n$ is 
well defined on a modification $\rho : W \ra \bp^n$ which is an isomorphism 
outside crossings of hyperplanes. In particular if $i : W \ra \bp^n$ is the 
extension of inversion then,
$$
i^* \cO_{\bp^n} (1) = \rho^* \cO_{\bp^n} (n) - E
$$
for $E$ an effective divisor on $W$ contracted by $\rho$. Applying this to
$if\vert_{\D (R)}$ gives in the notation of \ref{ex:compact}:
$$
\frac{1}{2} \int_{\vert z \vert = R} \log ( 1 + \sum_{i=1}^n 
\vert 
g_i \vert^{-2} ) \frac{d\theta}{2\pi} \leq n  T_f (R) -\log f^*\Vert\un_E\Vert(0)
+\frac{1}{2}\log ( 1 + \sum_{i=1}^n \vert g_i \vert^{-2} )(0)
$$
The definition of $N$ implies that the final term at the origin
is bounded by $\log(1+nN^2)^{1/2}$, which in any case is just
some constant $C(U)$ depending on $U$, and such a constant
also bounds the exceptional divisor as soon as $\rho(E)$ is
not in the closure of $U$.
A similar bound for 
$\log \left( 1 + \sum \vert g_i \vert^2 \right)$ is even easier, and so we 
obtain, for $\vert z \vert < R/2$,
$$
\left\vert \frac{g'_i}{g_i} \right\vert \leq 8 R^{-1} \left( (n+1) \, T_f (R) + C(U)\right)
$$
with $C(U)$ as per its definition. Now consider the Fubini-Study metric 
for the particular choice of coordinates, with $\om$ the corresponding $(1,1)$ 
form then,
$$
\om \leq 2 \sum_i 
\frac{dd^c\vert x_i\vert^2}{(1 + \vert x_i 
\vert^2)}
$$
so that plugging in our estimates for $g'_i$ gives the lemma.
\end{proof}
  
Let us apply this to establish the non-trivial
part of,

\begin{prop}\label{prop:closed} Let $f_n:\D\ra X$ be a
sequence of maps from the disc, and
suppose the former (area) alternative in \ref{eq:unbounded},
holds at some $r_0\in (0,1)$ with sufficiently rapid
growth, cf. \ref{claim:lang1}, in $n$ (so, always possible
after passing to a sub-sequence) then for $r \geq r_0$ outside a 
set of finite hyperbolic measure (i.e. $(1-r^2)^{-1} \ dr$) 
in $(0,1)$ any weak accumulation point of the $A_n (r)$,
\ref{eq:cint},
defines a closed positive current. Less obviously,
suppose $X$ is projective, the latter (Nevanlinna) alternative of
\ref{eq:unbounded} holds, and that the $f_n(0)$ converge,
then the same conclusion holds for the $T_n(r)$ of \ref{eq:cnev}.
\end{prop}
\begin{proof}
As we've said, the assertion for the $A_n(r)$
is the length-area principle, and is trivial,
{\it i.e.} for $\alpha$ a smooth $1$ form on $X$,
$$
\vert\int_{\D (r)} f_n^*d\a\vert= \vert\int_{\pa\D (r)} f_n^*\a\vert
\leq r \Vert \a \Vert \int_{\pa\D (r)} \vert df_n \vert_{\om} \frac{d\theta}{2\pi}  
\leq \Vert \a \Vert \left(r\frac{d}{dr}\int_{\D (r)} f_n^*\om\right)^{1/2}
$$
for a suitable determination of the sup-norm,
$\a\mpo \Vert \a \Vert$ independent of $f_n$,
$\om$ as per \ref{defn:fubini}, and the rest follows from:

\begin{claim}\label{claim:lang1} {\em Let $\d > 0$ and $S_n$ a sequence of increasing 
differentiable functions on $[0,1)$ with $S_n (r_0) \geq n^{2/\d}$ for some $r_0 
\in (0,1)$ then the set,
$$
\{ 1 > r \geq r_0 : S'_n (r) \geq S_n^{1+ \d} (r) \, (1-r^2)^{-1} , \ \hbox{for 
some} \ n \}
$$
has finite measure in the induced hyperbolic measure on the unit interval.}
\end{claim}
\noindent{\it Proof.} The set in question has measure bounded by,
$$
\sum_n \int_{r_0}^1 \frac{S'_n (r)}{S_n^{1+\d} (r)} \, dr \leq \sum_n \int_{S_n 
(r_0)}^{\ify} \frac{dx}{x^{1+\d}} = \frac{1}{\d} \sum_n 
\frac{1}{S_n (r_0)} \, \d = 
\frac{\pi^2}{6\d} < \ify\,\,\,\,\,\, \Box
$$
As such let's concentrate on getting a similar bound
in the Nevanlinna case, which will permit us to apply
\ref{claim:lang1}.  Again we begin with Stokes,
$$
\vert {\int\!\!\!\!\!\!\nabla}_{\D (r)} f_{n}^{*}d\a \vert = 
 \left\vert \int_{\D (r)} \, f_n^*\a 
d \log \vert z \vert \right\vert \leq \Vert \a \Vert\int_{\D (r)} \vert
 df_n \vert_{\om} \frac{dtd\theta}{2\pi}
$$
but now there is a problem close to the origin which requires care,
and we put:
$$J_n (s) = \int_{\vert z \vert = s} \vert df_n \vert_{\om}^2 \frac
{d\theta}{2\pi},\,\,\,\,\,
\hbox{so that,}\,\,\,\,\,
{\int\!\!\!\!\!\!\nabla}_{\D (r)}  f_n^*  \om   =  
\int_0^r J_n (s)  s \log \frac{r}{s}ds
$$
which we re-arrange as:
$$
I_n (r) := \int_0^r J_n (t)  t  \vert \log t \vert  dt = 
{\int\!\!\!\!\!\!\nabla}_{\D (r)} f_n^*  \om + \vert \log r \vert r \frac{d}{dr}  
{\int\!\!\!\!\!\!\nabla}_{\D (r)}f_n^* \om 
$$
Consider for a suitable $\e > 0$, depending on $n$, to be chosen,
\begin{eqnarray*}
\int_{r \geq \vert z \vert \geq \e} \vert df_n \vert_{\om} \frac{dtd\theta}
{2\pi} &\leq & \int_{\e}^r J_n (t)^{1/2} \, dt \\
&\leq & \left( \int_{\e}^r t \, \vert \log t \vert \, J_n (t) \ dt \right)^{1/2} 
\left( \int_{\e}^r \frac{dt}{t\vert \log t \vert} \right)^{1/2} \\
&\leq & I_n^{1/2} (r) \log^{1/2} \frac{\vert \log \e \vert}{\vert \log r \vert} \end{eqnarray*}
At which point take $r\geq r_0$, $n$ sufficiently large, and
use the $\ll$ notation for inequality up to a constant that depends on 
what we can fix, i.e. $r_0$, $X$, $\om$ etc. but definitely not $f_n$
or $r$. Then by \ref{lem:origin}, and in the notation of
the same, for
any $\log\e \ll -T_{f_n}(r_0)$ there is a constant, $C$,
independent of $r$ and $n$ such that:
\begin{eqnarray*}
\int_{\vert z \vert \leq r} \vert df_n \vert_{\om} \frac {dtd\theta}{2\pi} &\ll & 
I_n^{1/2} (r) \log^{1/2} \frac{\vert \log \e \vert}{\vert \log r \vert} + 
\e T_{f_n}(r_0)\, {\exp}(CT_{f_n}(r_0))\\
&\ll & I_n^{1/2} (r) \log^{1/2} \left( \frac{T_{f_n} (r)}{\vert \log r \vert} \right)
\end{eqnarray*}
after a suitably choice of $\e$ depending on $n$.
It only remains to show that $I_n (r)$ is not 
that much bigger than $T_{f_n} (r)$, which is
again \ref{claim:lang1}.
\end{proof}
 
\noindent This is the main lemma of this section, on which we'll
need some variants, {\it viz:}

\begin{sch}\label{sch:systems}
{\em 
Starting from a given projective variety $X$, we
will have need to work with various auxiliary (projective)
varieties $\pi_i:X_i\ra X$, $i\in I$ countable, or, indeed champs
de Deligne-Mumford with projective moduli, albeit the latter is 
only a technical convenience that one may eschew.
Consider attempting to lift the currents, $T_n(r)$ of \ref{eq:cnev}
to currents $T_{in}(r)$ on $X_i$. The 
maps $\pi_i$ will be proper, so
we have commutative diagrams:
$$
\begin{CD}
C_{in} @>{f_{in}}>> X_i \\
@V{p_{in}}VV @V{\pi_i}VV  \\
\Delta @>{f_n}>> X
\end{CD}
$$
with $C_{in}$ smooth 
(possibly  a champ if $X_i$ is)
and $p_{in}$ proper and finite. In order
to maintain functoriality, we don't exhaust the $C_{in}$
by their Green's function as in \ref{eq:nev}, but by
$p_{in}$. Thus for $r\in (0,1)$, $C_{in}(r)=p^{-1}(\D (r))$,
and,
\begin{align}\label{eq:lift}
 \nint_{C_{in}(r)}& :A^{1,1}(C_{in})\ra\bc:\tau \mpo \dfrac{1}{\mathrm{deg}(p_{in})}\cdot
\int_0^1\frac{dt}{t}\int_{C_{in}(r)} \, \tau \\
 T_{in} (r) &: A^{1,1} (X_i) \ra {\mathbb C} : \tau \mpo 
\left( \nint_{\D (r)} \, f_n^* \, \om \right)^{-1} {\nint}_{C_{in}(r)} \, f_{in}^* \, \tau \,
\end{align}
Now, by construction, $(\pi_i)_* (T_{in} (r))= T_n(r)$, but there
are the following issues to take care of:

\noindent{\bf (a)} Whether in \ref{fact:cart}, \ref{ex:compact}, or similar,
the point $p$ should be replaced with $p_{in}^{-1}(0)$, so control
at the origin for all $i$ and $n$ will still be achievable by control
of the whereabouts of the $f_n(0)$.

\noindent{\bf (b)} The $T_{in}(r)$ need not be bounded. Actually this
can only happen if $\pi_i$ is not finite. Indeed, 
$\pi_i$
has a Stein factorisation, $X_i\build\ra_{}^{\tau_i} Y_i \build\ra_{}^{\sigma_i} X$,
or better $X_i\build\ra_{}^{\rho_i}\vert X_i\vert 
\build\ra_{}^{\tau_i} Y_i \build\ra_{}^{\sigma_i} X$
the (finite) map to moduli followed by Stein factorisation
if $X_i$ is a champ, while for any finite map $Y\ra Z$ the difference
between a metric on $Z$ and $Y$ is $\ddc(\mathrm{bounded})$,
so, for example, when $\pi_i$ is finite $T_{in}(r)$ is
bounded by an exercise in integration by parts \`a la \ref{fact:cart}.
If, however, $\tau_i$ were non-trivial, then 
since everything is projective,
a metric $\om_i$
on $X_i$ is almost everywhere of the form:
$$\om_i = \tau_i ^* \eta_i + \ddc\log \Vert\un_E\Vert^2 $$
for $\eta_i$ a metric on $Y_i$ and $\Vert\un_E\Vert^2 $ 
some distance function on a subscheme $E$ blown down
by $\tau_i$. Consequently we can guarantee that the
$T_{in}(r)$ are bounded if the origins $f_n(0)$ are
bounded away from $\pi_i(E)$, or, just that
the distance between them satisfies an appropriate
growth condition.

\noindent{\bf (c)} The proof of \ref{prop:closed} could fail, and
limits of the $T_{in}$ need not be closed. Modulo a variant,
\ref{claim:lang2} of \ref{claim:lang1}, this could only
happen if the pre-image of discs under $p_{in}$ of 
the size
for which \ref{lem:origin} holds failed to be discs,
which, again, cannot happen provided that the origins
$f_n(0)$ are bounded away from the ramification of
$\pi_i$ or satisfy an appropriate growth condition on
the distance to it. Better, a further case, therefore,
in which one can dispense with any condition on the
distance between the ramification in $\pi$ and $f_n(0)$
is $\pi$ finite and $f_n$ not meeting any point of
the branch locus. 

The aforesaid, albeit required, variant of \ref{claim:lang1} is simply:}
\end{sch}
\begin{claim}\label{claim:lang2}
Let $S_{ni}$ be a 
 doubly indexed sequence 
of increasing differentiable functions
then for $r \in 
(0,1)$ outside a set of finite hyperbolic measure:
$$
S'_{nk} (r) \leq \{ S_{nk} (r) + 2e \} \log^{n^2 \, + \, k^2 \, + 
\, 1} \{ S_{nk} (r) + 2e \} \, (1-r^2)^{-1} \, .$$
\end{claim}
\noindent{\it Proof.} Indeed writing $F_{nk} (r) = S_{nk} (r) + 
2e$, we have $\log F_{nk} (0) > 1$, and the set in question has 
measure bounded by,
$$
\sum_{n,k} \int_{F_{nk} (0)}^{\ify} \frac{dx}{ x \log^{n^2 \, + \, 
k^2 \, + \, 1} x} = \sum_{n,k} \frac{1}{(n^2 + k^2) (\log F_{nk} 
(0))^{n^2 \, + \, k^2}} < \ify \, \hspace{0.8cm} \Box
$$

Whence to summarise:

\noindent{\bf Summary} \ref{sch:systems}.bis: Let $\pi_i:X_i\ra X$ be
as in \ref{sch:systems}, and $T$ a closed positive current arising
from a weak limit of some sub-sequence $T_n(r)$ afforded by discs
$f_n$ satisfying the conditions of \ref{prop:closed}, then provided
the $f_n(0)$ are bounded away from the branch locus of $\pi_i$, or,
indeed the distance between these satisfies an appropriate growth
condition, there is a closed positive current $T_i$ on $X_i$ such
that $(\pi_i)_*(T_i)=T$. In addition if $D_i$ is any effective Cartier
divisor on $X_i$ such that the $f_n(0)$ do not accumulate in $\pi_i(D_i)$,
or, again, a suitable growth condition on the distance between the
same, then: $D_i\cdot T_i\geq 0$.

Finally we need to compare the limits afforded by the
different definitions \ref{eq:cint} and \ref{eq:cnev}.
Observe that in the former case, the limiting discs
$\D_n$ satisfy:
\begin{defn}\label{defn:ahlfors}  The length
of the boundary $l_n$ in the metric $f_n^*\om$ is $o(a_n)$
of the area computed in the same, i.e. if
the area alternative of \ref{eq:unbounded} holds, the limits
of the $A_n(r)$ in \ref{prop:closed} are Ahlfors''
currents in the sense of \cite[pg. 306]{duval}.
\end{defn}

Now for $t\in (0,1)$, let $a_n(t)$ be the
area of the disc of radius $t$ with respect to
$f_n^*\om$, and $r_0$ the supremum of $t\in [0,1)$
for which $a_n(t)$ is bounded. We'll suppose
that $r_0<1$, so the Nevanlinna alternative
of \ref{eq:unbounded} holds for all $r>r_0$.
Fix such a $r$,  
then on $[0,r]$ we have probability measures:
\begin{equation}\label{eq:measures}
\begin{split}
\dm_n(t)& := \left(\nint_{\D (r)} f_n^*\om\right)^{-1} a_n(t)\frac{dt}{t},
\,\,\,\mathrm{and,}\\
\dn_n(s)& := \left(\nint_{\D (r)} f_n^*\om\right)^{-1} a_n'(s)\log\vert\frac{r}{s}\vert ds
\end{split}
\end{equation}
which for any continuous $\rho$ are related by:
\begin{equation}\label{eq:transform}
\int \rho\dm_n(t) = \int\dn_n(s) \frac{1}{\log\vert\frac{r}{s}\vert}\cdot
\int_s^r \rho(t)\frac{dt}{t} := \int\dn_n(s) N_r(\rho) (s)
\end{equation}
and the value at $s$ equal to $0$, or $r$, of the Nevanlinna type transform
$N_r(\rho)$ is just the above limited in $s$.

Similarly for $U\sbs X$ open, we have the area in
$U$, {\it i.e.}
\begin{equation}\label{eq:Uarea}
a_n^U(t):= \int_{f_n^{-1}(U)} f_n^*\om,
\,\,\,\,\, \hbox{together with the ratio:}\,\,\,
\phi_n(t) = \frac{a_n^U(t)}{a_n(t)}
\end{equation}
or, if one prefers to keep things smooth, multiply $\om$ by
a smooth $[0,1]$ valued bump
function identically 1 on $U$, and, in any case, there are measures:
\begin{equation}\label{eq:Umeasures}
\begin{split}
\dm^U_n(t)& := \left(\nint_{\D (r)} f_n^*\om\right)^{-1} a^U_n(t)\frac{dt}{t},
\,\,\,\,\, \mathrm{and,}\\
\dn^U_n(s)& := \left(\nint_{\D (r)} f_n^*\om\right)^{-1} 
\left( \int_s^r \phi_n(t) \frac{dt}{t}\right) a_n'(s)ds
\end{split}
\end{equation}
Subsequencing we may suppose that the $\phi_n$ converge
pointwise to some $\phi$, while \ref{eq:measures}
converge to probability measures $\dm$, 
$\dn$ on $[0,r]$, related as in \ref{eq:transform}, and \ref{eq:Umeasures} converge to measures
$\dm^U$, $\dn^U$, necessarily absolutely continuous
with respect to $\dm$, respectively $\dn$.
Let us observe:
\begin{lem}\label{lem:lebesgue} Notation as above, then:
\begin{enumerate}
\item[(a)] If $\dm$ has no support in $(a,b]\subseteq (0,r]$ then
$\dn$ has no support in $(0,b]$.
In particular if $r_0>0$ then none of the above are supported
in $(0,r_0)$.
\item[(b)] In the open interval $(0,r)\sbs [0,r]$ the 
Lebesgue derivative of $\dn^U$ by $\dn$ is, 
$N_r(\phi)$, defined exactly as in \ref{eq:transform},
at $s\in (0,r)$.
\end{enumerate}
\end{lem}
\begin{proof}
\ref{eq:transform} is valid for the characteristic
function of the interval too, which proves (a),
while,  by Ergoff's theorem, $N_r(\phi_n)\ra N_r(\phi)$ 
uniformally on compact subsets of $(0,r)$ which proves (b).
\end{proof}
From which we progress to:
\begin{fact}\label{fact:almost}{\em
Suppose the total mass of $\dm^U$, equivalently
that of $\dn^U$, is non-zero then if $\dn^U$
has support on $(0,r)$ there is a  $r_0< t < r$
such that, after subsequencing, 
$A_n(t)$ of \ref{eq:cint} 
converge to an Ahlfors current,  non-zero  in $U$.
Better still, there is a set, $E\sbs (0,1)$, 
depending on $\phi_n$, of zero Lebesgue measure
such that if $r\in (0,1)\bsh E$ 
is sufficiently large, then the same holds
whenever
$\dn^U$ has support
in $(0,r]$. }
\end{fact}
\begin{proof} The easy case occurs should $\dn^U$
have support in $(0,r)$ so that by \ref{lem:lebesgue},
$\phi\vert_{(r_0,r)}\neq 0$ on a set of positive
Lebesgue measure. Subsequencing doesn't change $\phi$,
but it does make the measure of the set occurring
in \ref{claim:lang1}, modulo scaling from $(0,r)$
to $(0,1)$ and putting $S_n=a_n$ as small as we please. 
Whence we can find $t\in (r_0,r)$ for which we 
simultaneously have length area for all $n$ on
discs of radius $t$ and $\phi_n(t)$ converging
to $\phi(t)\neq 0$. 

This also proves the better still, with $E$ empty,
provided $\phi\vert_{(r_0,1)}\neq 0$ almost everywhere,
and $r$ is sufficiently large. Consequently we
may suppose $\phi$ zero Lebesgue almost everywhere.
For $n\in\bn$ consider the decreasing functions:
$$\psi_n(t):=\sup_{m\geq n} \phi_m(t)$$
which still converge to zero Lebesgue almost
everywhere. On the other hand, \cite[2.9.19]{fed},
for Lebesgue almost all $r$, the limit,
$$
\lim_{s\ra r^{-}} \frac{1}{\log\vert\frac{r}{s}\vert} \cdot \int_s^r \psi_n (t)\frac{dt}{t}
$$
exists, and is equal to $\psi_n(r)$, so, zero for
$r\notin E$, with $E$ of zero Lebesgue measure.
As such, if $\e>0$ is given, and $r\notin E$,
there is a $\d>0$ such that for all $s\in (r-\d,r)$,
$N_r(\psi_n)(s)\leq \e$. By construction, $N_r(\psi_n)\geq N_r(\phi_m)$
for all $m\geq n$, so for such a $r$ the measure
of $\dn^U$ on $(0,r]$ is at most $\e$ which was
arbitrary, {\it i.e.} the assertion is vacuous
at Lebesgue almost all $r$.
\end{proof}
The utility of which will present itself in the
next section.

\subsection{Compactness}\label{SS:compact}

Although subordinate to more general theorems,
the following will prove useful,

\begin{lem}\label{lem:easycon} Let $f_n:\D\ra \bp^m$
be a sequence of discs defined up to the boundary
with $\om$ as per \ref{defn:fubini}, and admitting
a uniform bound in $n$ for,
$$
\nint_{\D} f_n^*\om
$$
then we may write $f_n=[f_{n0},\hdots,f_{nm}]$ in such
a way that after subsequencing the $f_{ni}$, $0\leq i\leq m$, converge uniformly
to some $g_i$ on compact subsets of $\D$. 
\end{lem}
\begin{proof} Without loss of generality we may suppose the $f_n(0)$ accumulate 
somewhere, so say $1 = [1,\ldots ,1]$ in an appropriate projective coordinate system 
consistent with \ref{defn:fubini} as already employed in \ref{lem:origin}.
Consequently 
if $X_i \in \Gamma (\bp^m , H)$ is the equation of a coordinate hyperplane then modulo 
subsequencing for any $r\in (0,1)$,  we can suppose that for 
each $i$, the $f_n^{-1} (X_i) \cap \D (r)$, counted with multiplicity are a convergent 
sequence of $0$ cycles in $\bar{\D}(r)$ which,
after slightly increasing $r$
if necessary, belong to $\bar{\D}(s)$ for some 
$s < r$, which we write as,
$$
f_n^{-1} (X_i) = \sum_{z \in \bar{\D}(s)} a_{ni} (z) [z] \ra \sum_{z \in \bar{\D}(s)} a_i 
(z)[z]
$$
where as in the proof of \ref{lem:origin}, $a_{ni} (z) = 0$ for all $\vert z \vert \leq t$, 
$t$ independent of $i$ 
and $n$. The disc being Stein we can write $f_n \mid_{\D (r)}$, for an 
appropriate $1 > R > r$ as,
$$
\left[ \prod_{w \in \bar{\D}_s} (z-w)^{a_{n0} (w)} , \prod_{w \in \bar{\D}_s} 
(z-w)^{a_{n1} (w)} u_1 , \ldots , \prod_{w \in \bar{\D}_s} (z-w)^{a_{nm} (w)} u_m \right]
$$
where the $u_1 , \ldots , u_m$ are units. Now consider the meromorphic functions $g_{ni}$ 
defined as $f_n^* \, X_i / f_n^* \, X_0$ and apply Jensen's formula, i.e. for $z \in 
\D (r)$,
$$
\log \left(\vert g_{ni} (z) \vert \cdot \vert P_{ni} (z)\vert\right) = \int_{\vert w \vert = R} \log 
\vert g_{ni} (w)\vert \, \mathrm{Re} \left\{ \frac{w+z}{w-z} \right\} d^c\log\vert w
\vert^2
$$
where $P_{ni} (z)$ is the Blaschke product,
$$
P_{ni} (z) = \prod_{w \in \bar{\D}(s)} \left\{ 
\frac{R^2 - \bar{w} \, z}{R(z-w)} 
\right\}^{a_{ni} (w) - a_{n0} (w)} 
$$
Consequently employing the technique of \ref{lem:origin}, or more correctly its proof, to bound the 
integrals over $\vert w \vert = R$ of $\log (1 + \vert g_{ni} \vert)$ and $\log \left( 1 + 
\frac{1}{\vert g_{ni}\vert} \right)$ independently of $n$ we conclude that for some 
constant $C$ independent of $n$, $z \in \D (r)$,
$$
\vert u_{ni} \vert (z) \leq C \cdot \frac{(R^2 - rs)^a}{R^b}
$$
for non-negative integers $a,b$ independent of $n$ determined by the signs of $a_i (w)$ 
-  $a_0 (w)$, so that the $u_{ni}$ converge on subsequencing and we're done.
\end{proof}

The lemma does not, imply,
however that the $f_n$ converge, {\it e.g.}
$$
f_n : \D \ra \bp^1 : z \ra \left[z - \z , z -\z + \frac{1}{n} 
\right]
$$
for any $\z\in\D\bsh 0$, and quite generally, a similar
problem presents itself at any point in the common 
zero locus of the limit $g_0 , \ldots , g_m$. On the other hand 
the maps $f_n$ manifestly converge, and 
constitute a particularly simple case of Gromov
convergence.
Deducing this from \ref{lem:easycon}
is postponed till it is necessary, \ref{prop:nobubbles},
and for the moment 
we summarise a more general setting, thus concentrating
on the differences arising from
using \ref{eq:nev} rather than \ref{eq:int},
beginning with:

\begin{defn}\label{defn:bubble} A disc with bubbles $\D^b$ 
is a connected $1$-dimensional analytic curve with singularities at 
worst nodes exactly one of whose components is the unit disc $\D$, 
and the closure of every connected 
component $R_z$, $z \in R (\D^b) := \D \cap {\mathrm{sing}} (\D^b)$, of the
complement of $\D$ is a tree of smooth rational curves.
\end{defn}

Clearly for $\tau$ an integrable $(1,1)$ form on each component
of some $\D^b$ without a bubble at the origin we
can extend \ref{eq:notation}, by way of,
\begin{equation}\label{eq:morenotation}
\nint_{{\D}^b_r}\,\tau:= \nint_{\D (r)} \, \tau + \sum_{\genfrac{}{}{0pt}{}{ z \, \in \, R (\D^b)}{\vert z \vert \, < \, r}} \log 
\frac{r}{\vert z \vert} \int_{R_z} \tau
\end{equation}
and similarly \ref{eq:cint}, which necessarily has 
similar positivity properties for 
intersections with effective divisors to the intersection product 
over discs. In addition we can define a graph $\G_f$ as,
$$
\G_f = ({\rm id} \ts f) \, (\D) \bigcup_{z \, \in \, R (\D^b)} z \ts f 
(R_z) \sbs \D \ts X \, .
$$
and we have, 
\begin{fact}\label{fact:misha}\cite{gromov}
{\em Let $X$ be a complex space admitting a K\"ahler
form $\om$ (or indeed champ admitting the same) then
if $X$ is compact any sequence of maps $f_n:\D\ra X$ such that,
we have a uniform bound in the Nevanlinna area,
$$
\nint_{\D} f_n^*\om
$$
admits a subsequence
converging to a disc $f:\D^b\ra X$ with bubbles,
{\it i.e.} the graphs $\G_{f_n}$ of compact sets
converge in the Gromov-Hausdorff metric to $\G_f$
on compact sets. Unlike the corresponding proposition
for,
$$
\int_{\D} f_n^*\om
$$
the converse is, in general, false, to wit: there may be discs
$f_n$ converging to a disc with bubbles such that,
$$
\nint_{\D (r)} f_n^*\om
$$
is unbounded for every $r>0$. Indeed, by \ref{eq:morenotation},
the Nevanlinna area is bounded in $n$ iff $f_n$ converge
to a disc with bubbles without a bubble in the origin,
equivalently, the $f_n$ converge normally in a neighbourhood
of the origin.}
\end{fact}

Alternatively, and by way of notation,
\begin{defn}\label{defn:hom}
Let $X$ be as in \ref{fact:misha}, then 
the space of maps $\mathrm{Hom} (\D , X)$
from the unit disc, may, whenever
$X$ is compact, be minimally compactified
by the space $\ov{\mathrm{Hom}}(\D , X)$
of discs with bubbles. Or, in the Nevanlinna context
pointed discs with no
bubbling at the point, and, in any case, we have a
strict inclusion,
$$
\mathrm{Hom} (\D , X) \sbs \ov{\mathrm{Hom}}(\D , X)
$$
if and only if $X$ (say smooth for safety) contains a rational curve.
\end{defn}

It is, of course, often convenient to move
the point, so for $C:(0,1)\ra \br$, and $K \sbs \mathrm{Aut} (\D)$ compact
another useful variant is that:
\begin{equation}\label{eq:misha}
\left\{ f \in \ov{\mathrm{Hom}} (\D , X) : 
{\int\!\!\!\!\!\!\nabla}_{\bar{\Delta}(r)} \, \a ^* f^* \, \om
\leq C(r) , \a \in K
\right\}
\end{equation}
is compact, so, in
particular $f_n \in \ov{\mathrm{Hom}} (\D , X)$ has a 
convergent subsequence iff for some subsequence $f_k$, radii 
$r_m \ra 1$ and $\a_k \in \mathrm{Aut} (\D)$ convergent, 
the degrees of the $ \a_k^* f_k $ at $r_m$ are 
bounded 
independently of $k$ for each $m$.

Consider the same with boundary,
{\it i.e.} $X$ compact, and an effective divisor:
\begin{equation}\label{eq:boundary}
B=\sum_i B_i,\,\hbox{
each component $B_i$ of which is $\bq$-Cartier}.
\end{equation}
where the key point is a lemma of Mark Green, {\it viz:}
\begin{lem}\label{lem:green} (\cite{green})  Let 
$(X,B)$ be as in \ref{eq:boundary}, and
$f_n : Y \ra X 
\bsh B$ be a sequence of maps converging uniformly on compact 
subsets of a Riemann surface to $f : Y \ra X$, then either, $f$ 
maps to $X \bsh B$ or $B$, so by induction, there is a
minimal (possibly empty) set of components $I$ such that
$f$ maps to $B_I:=\cap_{i\in I} B_i\bsh\cup_{j\notin I} B_j$.
\end{lem}
Thus if we understand by $\ov{\mathrm{Hom}} (\D , X\bsh B)$,
the closure of ${\mathrm{Hom}} (\D , X\bsh B)$ in
$\ov{\mathrm{Hom}} (\D , X)$, we have:
\begin{fact}\label{fact:boundary} {\em Let things be as in
\ref{eq:boundary}, and
suppose for every set of components $I$, 
$\cap_{i\in I} B_i\bsh\cup_{j\notin I} B_j$
contains no $\ba^1$'s, {\it i.e.} the Zariski
closure in $\cap_{i\in I} B_i$ is $\bp^1$ and
it's normalisation meets $\sum_{j\notin I} B_j$
in at most a point, then we have an identity,
$$
\left\{ \hbox{Closure of} \,\, {\mathrm{Hom}} (\D , X\bsh B)\,\,
\text{in}\,\, \mathrm{Hom}(\D,X)\right\}=\ov{\mathrm{Hom}} (\D , X\bsh B)
$$
and, conversely, if $X$ is smooth with simple 
normal crossing boundary and identity holds
then no stratum $B_I$
contains an $\ba^1$. }
\end{fact}
\begin{proof} 
The converse is for comparison with \ref{defn:hom}, and
is purely illustrative, so we ignore it. Otherwise,
we have a sequence $f_n : \D \ra X$ with area
uniformly bounded on 
compact sets, but not converging uniformly. As such there
is some $\z\in\D$
and maps $\vp_n : \D_{R_n} \ra \{ z \mid \vert z 
- \z \vert < \ve_n \}$ with $R_n \ra \ify$ and $\ve_n \ra 0$ such 
that $f_n \circ \vp_n$ converges uniformly on compact subsets 
to a map $f : \bc \ra X$ (cf. \cite{sik}),
of bounded area, so its closure is certainly
a $\bp^1$,
while 
the $f_n$ are maps to $X \bsh B$ so \ref{lem:green} applies
to $f_n\circ\phi_n$.
\end{proof}
By way of a variation on a  theme let us introduce:
\begin{defn}\label{defn:nobubbles}
Let $(X,B)$ be as in \ref{eq:boundary}, with
$f_n:\D\ra X\bsh B$ 
converging to $f:\D^b\ra X$ as per \ref{defn:bubble},
then we say that bubbles cannot form in the
boundary if for every $z$, $f(R_z)$ meets $B$ in at most
$f(z)$, and this only if $f(\D^b)\sbs B$.
\end{defn}
A close to optimal criteria for which is,
\begin{lem}\label{lem:green2}
Suppose for $I\neq\emptyset$ the strata $B_I$ do
not contain $\ba^1$'s then bubbles cannot form
in the boundary. 
\end{lem}
\begin{proof} We form the tree, $T$, whose vertices are
the components of $\D^b$ around $z$ with
vertices the intersections between them,
and root the tree in the disc component.
We aim to prove by decreasing induction 
on distance (in the graph metric) from the 
root that no component 
meets $B$, except possibly in the point
that corresponds to the vertex through 
which there is the unique path to the
root, and this only if
the disc component
maps to $B$ under $f$ and every
vertex in the path from it to the
component is contracted to $f(z)$.  

To this end, let $R_n\ni z$ be a small
disc such that $V_n=f_n(R_n) \sbs X\bsh B$ 
converge to $V=f(R_z)$. 
Now, quite generally, let us
suppose
that for some $x\in R_z$, $f(x)$
belongs to some component $B_j$. Should
$x\notin \mathrm{sing}(\D^b)$, then
$V_n\ra V$ uniformly close to $f(x)$,
and one argues as in \ref{lem:green},
{\it i.e.} if the component of $R_z\ni x$
is not contained in $B_j$ then the 
intersection of $B_j$ with $f(R_z)$ 
at $f(x)$ is the limit of that
with $R_n$ close to $f(x)$. This
works more generally, to wit:
if for $p\in V$, $c_1,\hdots,c_m$
are the components of $V$ through $p$,
then $p\in B_j$ iff some $c_k\sbs B_j$,
since: otherwise, there is a tubular
neighbourhood $U_j\supset B_j$ such
that,
$$V_n\cap U_j \ra V\cap U_j \in \mathrm{H}^{2\mathrm{d}-2}(U_j)$$
for $d$ the ambient dimension, and the
(nil) intersection number is conserved.

Now consider the graph,
$T_b$ obtained by colouring vertices
contracted to points by $f$ black,
white otherwise. Thus with the possible
exception of the root, all vertices
of valency 1, so, 
in particular
those at maximal distance from the
root,
may be supposed white. 
Observe that it's sufficient to
perform the induction for white
vertices, since otherwise a connected
component of blacks at greater distance mapping to a
point $p\in B_j$ would force some
white vertex to lie in $B_j$ by the
above considerations of local
intersection numbers. As such,
suppose we have a white vertex $v$
such that every white vertex at
a greater distance can at most meet
$B$ in the edge of its unique path
to $v$, and this only if the path
is black. At this point suppose
the set $I$ of boundary components containing
the vertex is non-empty, then for
$j\notin I$, as above, excepting the
edge leading to the root, no point of
the tree from $v$ on, {\it i.e.} at greater distance,
including the black vertices,
can meet $B_j$. Consequently the 
vertex $v$ gives an $\ba^1$ in
the stratum $B_I$, contrary to
our hypothesis.
\end{proof}

To which let us add some definitions
reflecting the different possibilities
that may occur in \ref{fact:misha},
\ref{fact:boundary}, and, \ref{lem:green2},
\begin{defn}\label{defn:bloch} Let $Z$ be a proper (and implicitly
without generic points in $B$) sub-variety
of a complex space $X$, then:

\noindent{\bf (a)} If a 
sequence of maps $f_n : \D \ra X$ of which a subsequence doesn't 
converge in $\ov{\mathrm{Hom}} (\D , X)$ is, in the compact open sense, 
contained in arbitrarily small neighbourhoods of $Z$, then
we say that $X$ is hyperbolic modulo $Z$.

\noindent{\bf (b)} If a sequence 
of maps $f_n : \D \ra X \bsh B$ of which a subsequence doesn't 
converge in $\ov{\mathrm{Hom}} (\D , X\bsh B)$ is, in the compact open sense, 
contained in arbitrarily small neighbourhoods of $Z \cup \cB$,
then we say that $(X,B)$ is hyperbolic modulo $Z$.

\noindent{\bf (c)} Everything as in (b), except that the $f_n$
are arbitrarily close to $Z$, then we say that $(X,B)$ is 
complete hyperbolic modulo $Z$, thus both this
item and (b) encompass (a) for $B=\emptyset$.

\noindent{\bf (d)} This is defined as
the following {\bf infinito} property for $(X,B)$:
every holomorphic map $f:\bc\ra X\bsh B$,
factors through $Z$.
\end{defn}
Now we can start to bring some order to the discussion, 
\begin{fact}\label{fact:nobloch}{\em
Suppose the Bloch principle fails, {\it i.e.} \ref{defn:bloch} (d)
does not imply \ref{defn:bloch} (b) or (c) as appropriate
({\it i.e.} for the moment either is permitted,
and for surfaces, \ref{summaryd:crit}, we'll give
necessary and sufficient algebraic criteria)  
and, \ref{defn:nobubbles}, that bubbles do not form in the boundary,
{\it e.g.} \ref{lem:green2} holds, 
then there
is a sequence $f_n:\D\ra X\bsh B$ of discs, unbounded
according to the Nevanlinna alternative \ref{eq:unbounded},
from some fixed radius on, such that:
\begin{enumerate}
\item[(a)] The origins $f_n(0)$ are bounded away from
$Z\cup B$ if \ref{defn:bloch}.(b) fails, respectively $Z$ if 
\ref{defn:bloch}.(c) fails.   
\item[(b)] For $r\in (0,1)$ outwith a set of finite
hyperbolic measure, and possibly after subsequencing,
any (weak) accumulation point $T$ of the $T_n(r)$ in
\ref{eq:cint} is supported on $Z\cup B$- actually
$Z\cup W$ for $W\sbs B$ the Zariski closure of the
union of
$\bc$'s (not just the excluded $\ba^1$'s) in boundary strata
\`a la \ref{fact:boundary}.
\item[(c)] If furthermore $X$ is projective, we
may suppose, \ref{prop:closed}, that $T$ is closed, and, of course,
by (a) and \ref{fact:cart}, $D\cdot T\geq 0$,
for every effective $\bq$-Cartier divisor, $D$, supported
in $Z\cup B$, should \ref{defn:bloch}.(b) fail,
respectively $Z$, if \ref{defn:bloch}.(c) fails.
\end{enumerate}
}
\end{fact}
\begin{proof}
The fact that we can choose the origins $f_n(0)$
as in (a) follows from \ref{eq:misha}, while
(c) is just a re-statement of things that have
already been proved. As such, if $T$ is
the resulting current, the new statement
is (b), and 
we require to prove that $T$ has no mass
off $Z$ or $Z\cup B$ as appropriate.
To this end,
one takes the $U$ of \ref{eq:Uarea}
to be 
either one of a relatively compact sequence
exhausting
$X\bsh (Z\cup B)$, respectively $X\bsh Z$,
or better replace $\un_U$ in op. cit. by
a continuous $[0,1]$ valued function vanishing
on exactly $Z\cup B$, respectively $Z$.
By \ref{fact:almost}, with notation as
therein, 
as soon as $\dn^U$ has support on $(0,r]$
we find an Ahlfors' current, $A$, with
mass on $U$ and so by
the main theorem
of \cite{duval} 
there is a $\bc$ with bubbles in $X$ meeting
$U$ which arises as a limit of discs
mapping to $X\bsh B$. By hypothesis bubbles cannot form
in the boundary, so, either the $\bc$ goes
into the boundary and the bubbles lie
in $U$ contradicting \ref{defn:bloch}.(d),
or the entire limit lies in $U$ contradicting
\ref{defn:bloch}.(d) again.  

We may, therefore, put ourselves in the
situation of \ref{fact:almost}, and suppose that $\dn^U$ is
supported uniquely in the origin, with, as
per the proof of \ref{fact:almost}, the
limit $\phi$ of the ratios, $\phi_n$, of
\ref{eq:Uarea}, converging to zero Lebesgue
almost everywhere. Equally, we may suppose that
the areas $a_n(t)$ of every compact
$t>0$ are unbounded in $n$, since otherwise,
by \ref{eq:morenotation} the mass is uniquely attributable to a
bubble at the origin, which must belong to
$Z$ or some boundary strata by \ref{defn:bloch}.(d),
and \ref{fact:boundary}, but equally cannot contribute
mass in $U$ since $f_n(0)$ is bounded away
from the bubbling locus. 

As such, given $\a>0$, for $n$ sufficiently
large, there is some $\e_n(\a)$, such that,
$a_n(\e_n(\a))=\a$. Now consider for each $\a$,
$$
\limsup_n\left(\nint_{\D (r)} f_n^*\om\right)^{-1}
\int_0^{\e_n(\a)} \phi_n(t) a_n (t) \frac{dt}{t}
$$
and suppose for some $\a>0$ this is non-zero.
However, for every $x\in (0,1]$,
$$
\int_0^{\e_n(\a)} \phi_n(t) a_n (t) \frac{dt}{t}
=
\int_0^{x\e_n(\a)} \phi_n(t) a_n (t) \frac{dt}{t} 
+
\int_{x\e_n(\a)}^{\e_n(\a)} \phi_n(t) a_n (t) \frac{dt}{t}
$$
and the latter integral is at most $\a\vert\log\vert x\vert\vert$,
so normalising the discs $\D_{\e_n(\a)}$ to radius
1, every compact $\D_x$ converges to a disc with
bubbles with mass in $U$, so by \ref{eq:morenotation},
without loss of generality a non-trivial bubble
uniquely at the origin. We may, however, suppose that $U$
is relatively compact in the complement of $Z$,
so such a thing is impossible.

Before we can profit from the above discussion,
observe that if $E\sbs (0,1)$ is any set of
finite $t^{-1}dt$ measure, then by Ergoff's
theorem the convergence of,
$$
\int_s^r \un_E \phi_n(t) \frac{dt}{t} \ra \int_s^r \un_E \phi(t) \frac{dt}{t}=0 
$$
is uniform in $s\in [0,r]$. Whence if $E'$ is
the complement of such an $E$, $\d>0$, and $\a$
are given, we may suppose that:
$$
\limsup_n\left(\nint_{\D (r)} f_n^*\om\right)^{-1}
\int_{\frac{\e_n(\a)}{\d}}^r 
\left[\frac{\phi_n(t\d) a_n (t\d)}{a_n(t)}\right]
\un_{E'}(t\d)  a_n(t)\frac{dt}{t}
$$
is some $2\eta\in (0,1)$ which is independent of $\d,\a$, and $E$.
In particular, for each $\a$ there are discs of radius
$t_n\in (0,\d r)\bsh E$ of area at least $\a$ and area
in $U$ at least $\eta\a$. Passing to a sub-sequence
indexed by $n=n(\a)$ with 
$\a$ growing rapidly, we 
can as in \ref{claim:lang1} choose $E$ a priori, so
that:
$$
ta_n'(t) \leq {(a_n)^{1+\epsilon}},\,\, \forall\, t\in(\e_n(\a), \d r)\bsh E  
$$
for some fixed $\epsilon\in (0,1)$, and since the length of
the boundary is bounded by the square root of the
left hand side, the discs $\D_{t_n}$ again yield
an Ahlfors current with mass in $U$, and 
we conclude once more by \cite{duval}.
\end{proof}

\subsection{Tautological Inequality}\label{SS:taut}

An appropriate level of generality that covers
both \ref{eq:green} and \ref{sch:systems} is
simply to replace the Dirac delta at $p$, by
a finite sum $\d_P:=\sum_i w_i \d_{p_i}$, $w_i >0$,
of total mass 1 in \ref{eq:green}, and we return
to the notation of \ref{SS:area} for the implied
exhaustion of a (regularly bordered) Riemann-surface by $g$.
Thus we have local functions $z_i$ around $p_i$
such that \ref{eq:z} holds for every $i$ on multiplying
the right hand side by $w_i$, while
the Euler characteristic in the
Nevanlinna sense of \ref{exd:euler} reads,
\begin{equation}\label{eq:euler}
E_\Si(r):=\frac{1}{2}\cdot \int_{-\infty}^0 \chi(\Si_t\bsh \vert P\vert) \, 
dt \,+\, \sharp(P) \frac{r}{2}
\end{equation}
for $\vert P\vert$ the support of $P$, and $\sharp(P)\in\bn$ its
cardinality. The example \ref{exd:euler} may
be generalised as follows:

\begin{fact}\label{fact:taut1} {\em Let $(X,\om_X)$ be a
complex space, or for that matter analytic champ
de Deligne-Mumford, with metric, and $\ov L$ the
resulting metricisation of the tautological 
bundle on $\bp(\Om^1_X)$ (EGA notation) then for
$f:\Si\ra X$, with derivative, 
$f':\Si\ra \bp(\Om^1_X)$,
unramified at $P$,
$$
\nint_{\Si_r} f^* \mathrm{c}_1(\ov{L}) 
+ E_\Si(r) + \mathrm{Ram}_f (r) + 
\sum_i w_i\log\frac{1}{w_i}\Vert f_*\frac{\pa}{\pa z_i}\Vert 
(p_i) 
=
\int_{\pa\Si_r} \log \Vert\pa\Vert_{\om_X} 
\dc g 
$$
for $z_i$ as above, and
for $\mathrm{ram}_f$ the ramification divisor,
$$
\mathrm{Ram}_f (r) \,:=\, \frac{1}{2}\cdot \int_{-\infty}^r \sharp \left(
\mathrm{ram}_f\cap \Si_t\right)
$$
}
\end{fact} 
\begin{proof} Define a function $\vert F\vert$ by,
$$
f^*\om_X = \vert F\vert^2 dgd^c g
$$
then $f^* \mathrm{c}_1(\ov{L})$ is $\ddc\log \vert F\vert^2$
Lebesgue a.e. on $\Si$ and one integrates by parts.
\end{proof}

Now the right hand side of \ref{fact:taut1} certainly
admits the bound,
\begin{equation}\label{eq:length} 
\log\bigl[(\mathrm{length}_{\om_X})(\pa\Si_r)\bigr]
\end{equation}
which if $X$ is compact is wholly negligible in
practice by \ref{claim:lang1}. However, if we
pass to a divisorial boundary $(X,B)$ \`a la \ref{eq:boundary},
but $X$ smooth at every point of $B$ with the latter
simple normal crossing divisor, so that $X$ of
\ref{fact:taut1} becomes $X\bsh B$, then the estimability
of \ref{eq:length} is wholly dependent on how the
(Nevanlinna) area computed in $\om_{X\bsh B}$ compares with that
computed in a smooth metric $\om_X$. For example,
if $x_1\hdots x_p=0$ is a local equation for the
boundary and $y_j$ the other coordinates then
a metric of the form,
\begin{equation}\label{eq:fucked}
\sum_i \frac{\ddc\vert x_i\vert^2}{\vert x_i\vert^2}
+ \sum_j \ddc\vert y_j\vert^2
\end{equation}
on $X\bsh B$, affords an area that has no a priori relation
with the same on $X$. On the other hand if we take
a complete metric on $X\bsh B$, {\it i.e.} everywhere locally commensurable
to,
\begin{equation}\label{eq:complete1}
\sum_i \frac{\ddc\vert x_i\vert^2}{\vert x_i\vert^2\log^2 \vert x_i\vert^2}
+ \sum_j \ddc\vert y_j\vert^2
\end{equation}
Then there is a constant $N>0$ such that,
\begin{equation}\label{eq:complete2}
\om_X \leq \om_{X\bsh B} \leq N\om_X - \sum_i \ddc\log\log^2\Vert\un_{B_i}\Vert^2
\end{equation}
supposing as we may that the sup of any $\Vert\un_{B_i}\Vert^2$ is
at most $e^{-1}$. Consequently for any $f:\Si\ra X$ which we allow to
meet $B$, but say $P\cap f^{-1}(B)$ for convenience, an
integration by parts together with our bound on the distance
to the $B_i$ gives,
\begin{equation}\label{eq:complete3}
\nint_{\Si_r} f^*\om_{X\bsh B} \leq
 N \nint_{\Si_r} f^*\om_{X} 
+ \log\left\vert\log\Vert\un_{B_i}\Vert^2 \right\vert (\d_P)
\end{equation}
so that controlling the distance of $P$ to $B$ gives
control on the area computed in the complete metric,
from which, the utility of:
\begin{fact}\label{fact:taut2}
{\em Let $(X,B)$ be compact with smooth simple normal crossing
boundary $B$, and $L^c$ a metricisation of the tautological
bundle of $\bp(\Om^1_X(\log B))$ resulting from
a complete metric  $\om_{X\bsh B}$ on $X\bsh B$, then
for any $f:\Si\ra X$ with logarithmic derivative,
$f':\Si\ra \bp(\Om^1_X(\log B))$, unramified at $P$,
and $f(P)$ missing $B$,
\begin{equation*}
\begin{split}
\nint_{\Si_r} f^* \mathrm{c}_1(L^c) 
+ \, & E_\Si(r) \,+\, \mathrm{Ram}^B_f (r)\, + 
\sum_i w_i\log\left[\frac{1}{w_i}\Vert f_*\frac{\pa}{\pa z_i}\Vert_{\om_{X\bsh B}}\right] 
(p_i) 
=\\
& \mathrm{Rad}^B_f (r)\, + \int_{\pa\Si_r} \log \Vert\pa\Vert_{\om_X} 
\dc g 
\end{split}
\end{equation*}
for $z_i$, $\mathrm{Ram}^B_f (r)$  as in \ref{fact:taut1}, 
except that one integrates the part
$\mathrm{ram}^B_f$ of the ramification divisor which
is not supported in $B$, while for $\mathrm{rad}^B_f$
the divisor $f^{-1}(B)$ counted without multiplicity,
$$
\mathrm{Rad}^B_f (r) \,:=\, \frac{1}{2}\cdot \int_{-\infty}^r \sharp \left(
\mathrm{rad}^B_f\cap \Si_t\right)
$$
}
\end{fact}
\begin{proof} Exactly as per \ref{fact:taut1}.\end{proof}
While \ref{fact:taut2} together
with the estimate \ref{eq:complete3} is
the most accurate reflection of the geometry,
it's not exactly admissible to apply something
like \ref{prop:closed} to the derivatives with
the complete metric rather than a smooth metric,
$\ov L$. Fortunately, a comparison is quite
easy, since:
\begin{equation}\label{compare}
\mathrm{c}_1(\ov{L}) = \mathrm{c}_1(L^c) + \ddc(\psi),
\,\,\, \text{where,}\,\,\, 0\leq \psi \leq \log\log^2\Vert\un_B\Vert^2
\end{equation}
and so we obtain,
\begin{fact}\label{fact:taut3}{\em
Let everything be as in \ref{fact:taut2} with $\ov L$ as above, then
for a constant $C$ depending only on the metrics:
$$
-\log\left\vert \log\Vert\un_B\Vert^2\right\vert(\d_P)\leq 
\nint_{\Si_r} (\mathrm{c}_1(\ov{L}) - \mathrm{c}_1(L^c))
\leq \log [ \vert \log\Vert\un_B\Vert^2\vert(\d_P) + C\nint_{\Si_r}\om_X]
$$
}
\end{fact}
\begin{proof} One integrates by parts, and the lower
bound is the easier one, with estimation exactly as
in \ref{eq:complete3}, while an upper bound is,
$$
\int_{\pa\Si_r} \log \left\vert \log\Vert\un_B\Vert^2\right\vert d^c g
\leq \log \int_{\pa\Si_r} \left\vert\log\Vert\un_B\Vert^2\right\vert d^c g
$$
and one concludes by \ref{fact:cart}.
\end{proof}
We conclude the pre-liminaries,
by observing that whether in \ref{fact:taut1}
or \ref{fact:taut2} we require some control
on the derivative at the origin, which we 
do by:
\begin{lem}\label{lem:Dorigin}
 Suppose a sequence $\{ f_n \} \in \mathrm{Hom} (\D , X)$ 
is given which is not eventually arbitrarily close 
(in the compact open sense)
to a Zariski closed subset 
$Z$ of $X$, and that no subsequence of $\{ f_n \}$ converges then after 
subsequencing there is a convergent sequence $\a_n \in \mathrm{Aut} (\D)$ such that $\mathrm{dist
} 
(\a_n^* \, f_n(0),Z)$ and $\left\Vert (\a_n)_* \, (f_n)_* \left( 
\frac{\partial}{
\partial z} \right) \right\Vert_{\om_X} (0)$ are bounded away from zero independently of 
$n$.
\end{lem}
\begin{proof}
 By hypothesis there is a disc of radius $r < 1$, with $f_n 
(\ov{\D}_r) \not\sbs U$ for $U$ some open neighbourhood of $Z$, and $\vert d 
f_n \vert_{\mathrm{euc}}$ large somewhere on $\ov{\D}_r$. Provided we simply take 
automorphisms that move the origin to somewhere in $\ov{\D}_r$ these will converge 
after subsequencing, so we might as well say that $\vert d f_n \vert_{\mathrm{euc}} 
(0)$ is as large as we like. So consider a connected component $U'_n$ of points in 
${\D}_r$ which are a suitable distance $\e$ away from $Z$, which contains 
an open subset of points $U''_n$ of points a distance $2\e$ away.  
Now choose $u_n\in U''_n$, and consider the ray $R\ni u_n$ of constant
argument joining $u_n$ to the origin.
Without loss of generality, $0\notin U'_n$, so
the connected
component, $\g$, of $U'_n\cap R$ containing $u_n$ has
some boundary (going towards the origin) at $v_n$
with $\mathrm{dist} (f_n (v_n) , Z) = \e$. In 
particular $\mathrm{dist} (f_n (u_n) , f_n(v_n)) > \e$, so:
$$
\e \leq \int_\g \Vert (f_n)_*(\dz)\Vert_{\om_X} \vert dz\vert
$$
and we find points in $\g$ ($\sbs U'_n$ by definition) of derivative at least $\e/r$. 
\end{proof}

\section{Algebraic surfaces}\label{S:surfaces}

\subsection{Minimal models}\label{SS:mmp}

The log-minimal model programme for quasi-projective
surfaces will afford optimal conditions for 
avoiding bubbles in the boundary as required
by \ref{fact:nobloch}. Starting from a smooth
projective surface $S$ with simple normal crossing
boundary $B$, one runs the minimal model programme
for $K_S+B$. The result $\rho : (S,B) \ra (S' , B')$
either has $K_{S'}+B'$ nef., or $S\bsh B$ is
covered, \cite{keelmck}, by $\ba^1$'s, and
is the opposite of what we're interested in.
At the initial stage in the programme any
$-1$ curve in $S$ curve that does not meet $B$
can always be contracted, so, we may simplify
the discussion by supposing that there are
no such. As a result an initial move in the programme
must involve a curve that meets the boundary,
and it always reduces the obstruction to
bubbling in the boundary.
Such curves need not, however, be contained
in the boundary, so $\rho^{-1}(B')$ can be
strictly bigger than $B$. Further, our notation,
\ref{eq:boundary} only ever permits divisorial
boundaries, and it can perfectly well happen
that a connected component of $B$ is contracted to a
point. In this latter case the correct thing
to do is to fill the point, and omit it from
the boundary. Indeed if, for example, $B$ were a -1 curve
so that $S'$ would be the contraction of the
same to a point $p$, then by \cite{kobcodim2}
discs converge in $S'\bsh p$ iff they converge
in $S'$, but a property such as \ref{defn:bloch}.(d)
for $S'\bsh p$ does not imply the same for
$S'$. The behaviour of discs in the general case of a boundary
component contracting to a point is wholly
identical up to the action of a finite
group, {\it i.e.} the resulting point
$p\in S'$ is an isolated quotient singularity,
and discs converge in $S'\bsh p$ iff they converge
in $\cS'_p$, where $\cS'_{p}\ra S'$ is the minimal
champ de Deligne-Mumford over $S'$ which is smooth
at $p$, or, in a perhaps more familiar language,
the filling at $p$ is by way of an orbifold point.

On the other hand, $(S', B')$, or better
a yet smaller object, the canonical model,
is where the natural geometry of the pair
$(S,B)$ is to be found, {\it e.g.} K\"ahler-Einstein
metrics are metrics on the canonical model,
and whence on $(S,B)$ reflect both the
above phenomenon, {\it i.e.} they are not
complete at components of $B$ which are
 contracted to quotient singularities,
but they will be complete around certain
divisors in the interior $S\bsh B$ whenever
$\rho^{-1}(B')$ is strictly bigger than
$B$. As a result, hyperbolicity questions
about smooth projective surfaces $S$ with
simple normal crossing divisor $B$ are
somewhat ill posed. Nevertheless,   
the canonical model, naturality
notwithstanding, may not be right for the
hyperbolicity problem that one wishes to
understand, and so we make:

\begin{factd}\label{factd:vistoli} 
{\em In the first place if $X$ is a normal variety, or even just
complex space, with quotient singularities,
then there is a unique, \cite{angelo}, champ de Deligne-Mumford, $\cX$,
the Vistoli covering champ,  
on $X$ ({\it i.e.} with the same moduli) such
that $\cX$ is smooth and $\cX\ra X$ almost \'etale.
Thus, if $(S',B')$ is the minimal 
model of some quasi-projective algebraic 
surface $(S,B)$, there is such a champ
$\cS'\ra S'$, a general pair $(S'',B'')$ 
with log-canonical singularities may, however,
have an isolated set $P\sbs S''$ 
of elliptic Gorenstein singularities (which can
only occur on contracting a connected component
of $B'$ which is numerically elliptic) so there
is only a champ $\cS''$ on $S''$
equal to the Vistoli covering champ (whence smooth) 
over
$S''\bsh P$, and an isomorphism around $P$
otherwise. 
In the case that $(S'',B'')$ is the canonical model,
we also distinguish an intermediate
object $(S',B') \ra (S_0,B_0)  \ra (S'',B'')$,
which will be said to have a canonical
boundary, if it is obtained from $S'$ by modifying 
in the boundary alone, and every 
generic point of $B_0$ and $B''$ coincide, so, in particular
it has the same set, $P$, of elliptic Gorenstein
as $S''$, and so there is a champ $\cS_0\ra S_0$
with the same properties.
In all cases the boundaries, $\cB'$, $\cB''$, and $\cB_0$
in the respective champ are,
\'etale locally, simple normal crossing, so,
inter alia $P\cap B_0=\emptyset$.}
\end{factd}

The relation of the above with the metric structure
has largely already been described since the
singularities of $S''$ remain isolated quotient
except in the elliptic Gorenstein case. Here,
the
K\"ahler-Einstein
metric is complete in the complement of $P$, 
and, indeed, modelled as in \ref{eq:complete1}
around the fibre $B'_P$.
Even locally, around $P$, however,
$B'_P$ is
an obstruction to completeness in
the sense of \ref{defn:bloch}.(c), albeit
never an occasion of bubbling in the 
boundary, cf. \ref{lem:green2}, so
that we can only expect such completeness
in $S_0$ or $S''$.
Let us
therefore make appropriate changes to
\ref{defn:bloch} to reflect this situation,

\begin{summaryd}\label{summaryd:crit} {\em (\ref{defn:bloch}.bis)  If $X$
of \ref{defn:bloch} is smooth then it matters
not a jot to any of properties (a)-(c) as to 
whether the maps $f_n$ take values in $X$ or
omit a co-dimension 2 subset, and, idem,
if $X$ were a smooth champ. As such if 
$(S',B')$ is the minimal model of a smooth quasi-projective
surface $S\bsh B$, or is the model $(S_0,B_0)$
with canonical boundary, then for $Z$ as per op. cit.: 

\noindent{\bf (b)} If a sequence of maps $f_n:\D\ra S'\bsh B'$
of which a sequence doesn't converge in 
$\ov{\mathrm{Hom}}(\D,S'\bsh B')\subseteq \ov{\mathrm{Hom}}(\D,\cS')$
is (compact open sense) contained in 
arbitrarily small neighbourhoods of $Z\cup B'$,
then we say that the minimal model of the 
smooth quasi-projective surface $S\bsh B$,
or $(S',B')$ if this is simply a pair
with klt singularities, and there's no danger of confusion,
is hyperbolic modulo $Z$.

In addition, a canonical boundary is a necessary
condition for completeness. Indeed,
if, say, $\cB_i\sbs \cB'$
were the champ over a boundary component $B_i$, the Euler characteristic
of $\cB_i\bsh\{\mathrm{sing}(\cB)\}$ is $-(K_{S'}+B')\cdot B_i$.
As such if this is zero, the universal cover of,
$\cB_i\bsh\{\mathrm{sing}(\cB)\}$ is $\bc$, so
arbitrarily small neighbourhoods of such contains
discs that are as big as one pleases. Whence,

\noindent{\bf (c)} If a sequence of maps $f_n:\D\ra S_0\bsh \{B_0\cup P\}$,
for $P$ the (possibly empty, always isolated) elliptic Gorenstein locus
of which a sequence doesn't converge in
$\ov{\mathrm{Hom}}(\D,S_0 \bsh \{B_0\cup P\})\subseteq \ov{\mathrm{Hom}}(\D,\cS_0)$
is, in the compact open sense, contained in
arbitrarily small neighbourhoods of $Z$,
then we say that the model of the smooth
quasi-projective surface $S\bsh B$ with canonical boundary,
or $(S_0,B_0)$ if this is simply a pair with
log-canonical singularities, and, there is no danger of confusion,
is complete hyperbolic modulo $Z$.

Of course, as above, strict positivity of $K_{S_0}+B_0$
on the boundary is a necessary condition for 
complete hyperbolicity, but, equally this always
holds on the canonical model in the only case where the definitions
can be non-empty, {\it viz:} $S\bsh B$ of
general type, while, finally, \ref{defn:bloch}.(d) is changed to,

\noindent{\bf (d)} If every holomorphic map,
$f:\bc\ra \cS'\bsh B'$ to the covering champ
of \ref{factd:vistoli}, respectively
$\cS_0\bsh \{B_0\cup P\}$, factors through 
$Z$ then
we say that {\bf infinito} or GG (Green-Griffiths) holds
for the minimal model, respectively
model with canonical boundary, of 
the smooth quasi-projective surface
$S\bsh B$, or even just that (GG)
holds for  $(S',B')$, respectively
$(S_0,B_0)$, if this is simply a
pair with klt, respectively canonical,
singularities, and there is no danger of
confusion. So, trivially, (GG)
always follows from (b) or (c)
as appropriate.  
}
\end{summaryd}

Modulo an issue about components of the
boundary that may have nodes, the above
considerations on the Euler characteristic
show that both $\cB'$ and $\cB_0$ have
\'etale neighbourhoods in which the
strata, in the sense of \ref{lem:green},
do not contain $\ba^1$'s. Unfortunately,
there's no reason for this \'etale covering
to extend over $\cS'$, so \ref{lem:green2}
does not apply to conclude no bubbling
in the boundary. Nevertheless,

\begin{prop}\label{prop:nobubbles}
Let $(S',B')$ be a minimal model of
a quasi-projective surface $S\bsh B$,
and $f_n:\D\ra \cS'\bsh B'$ (or,
better, $\D\ra S'\bsh \{B'\cup\mathrm{sing}(S')\}$,
since, it's easier, and, as we've observed the result
is the same) converge to a disc with
bubbles in $\cS'$ (albeit here non-schematic
points in the interior cannot be omitted),
then, \ref{defn:nobubbles}, the bubbles do not form in the
boundary.
\end{prop}
\begin{proof} We use \ref{lem:easycon}, and the notation
therein, so $f_n=[f_{n0},\hdots,f_{nm}]\in\bp^m$ with
$f_{ni}$ converging uniformly on compacts to some $g_i$. As such
if a bubble forms about $\zeta\in \D$ it's uniquely 
because all the $g_i$ vanish in $\zeta$, to some
order $p_i>0$. Replacing $\D\ni \zeta$ by a small
disc centred on the origin, and subsequencing
as necessary, we can say that $f_{ni}$ vanishes
at points $t_{nij}\ra 0$, $1\leq j\leq J_i$, to
order $q_{ij}$ with $J_i$, $q_{ij}$ independent
of $n$, and, of course $p_i=\sum_j q_{ij}$. 
So, in projective coordinates, close to the bubble:
$$
f_n: z\mpo \left[(\prod_{j\in J_0} (z+c_{n0j})^{q_{0j}}) u_{n0}(z),\hdots,
(\prod_{j\in J_m} (z+c_{nmj})^{q_{mj}}) u_{nm}(z)\right]
$$
with $u_{ni}$ units converging to units $u_i$,
and we may as well suppose $u_i(0)=1$. As
such we can write,
$$
f_n(z) = [P_0(z,t_{n0})u_0(z),\hdots, P_m(z,t_{nm})u_m(z)]^{\a_n(z)}
$$
For some $\a_n(z)$ in $\mathrm{PGL}_{m+1}$ identified
with a diagonal matrix in the image of the exponential
from $\mathfrak{pgl}_{m+1}$, {\it i.e.} close to the
identity, and $P_i(z,t_{ni})$ a close to monic
polynomial of degree $p_i$. The images of the 
 $P_i(z,t_{ni})$ in the space of such polynomials,
have, for each $N\in\bn$ a Zariski closure of
the set of such with $n\geq N$, which by Noetherian
induction must stabilise, and it is in this sense
that we understand the Zariski closure $T$ in
the space of polynomials.

From which, it's almost automatic that the
$f_n$ converge in $\bp^m$ with at worst bubbles,
{\it i.e.} up to a small perturbation our sequence
belongs to one of finitely many Hilbert schemes
compactifying rational maps to $\bp^m$. Our
proposition, however, regards $S'$, which for
ease of notation we'll suppose equal to $S$,
so we need a similar sort of interpolation
but in $S$ rather than just $\bp^m$. To this
end, suppose that some generic projection
$p:S\ra \bp^2$ has been chosen a priori,
so, without loss of generality the limiting
(small) disc with bubbles only meets the ramification, $R$,
in smooth points, and, again keeping the discs around
the point of bubbling sufficiently small,
we can suppose that, 
$$(f_n)^*R=\sum_{k\in K} e_k[c_{nk}]$$ 
are supported in some compact with $e_k$,
$K$ independent of $n$. In addition there is
some $N$, independent of $n$, such that
if some map $g$ agrees with $p\circ f_n$ to 
order $N$ at some $c_{nk}$ then around
the same $g$ lifts to $S$ with the same
order around $R$. Now we do what we did
before for the initial polynomial, but
this time for the automorphisms $\a_n$,
which we identify with a vector of
functions $h_{ni}$ vanishing at the
origin, and converging to zero uniformly.
At each $c_{nk}$ these have a Taylor
expansion to order $N$, giving a 
projection $a_n$ of each $\a_n$ to (a bounded 
subset) of $\ba^{NK}$. Again, in 
the above $\liminf$ sense we take the
Zariski closure of our sequence to
get some space of automorphisms $\a(z,a)\in A\sbs \mathrm{Hom}(\D, \mathrm{PGL}_3)$,
which are just exponentials of diagonal
polynomial matrices,
and which satisfy,
$$
\a(z,a_n) = \a_{n}(z)\,\, \text{mod}\,\, \mathfrak{m}_{c_{nk}}^N,
\,\, \forall\, 1\leq k\leq K$$
Identifying $A$ with the subset of
polynomials affording it, and now
taking the Zariski closure, $V$, of our
sequence $v_n=(a_n,t_n)\in A\times T$, we 
have maps,
$$
f(z,a,t):\D\times V(\sbs A\times T)\longra\bp^2: z\mpo
[P_0(z,t)u_0(z),\hdots, P_2(z,t)u_2(z)]^{\a(a,z)}
$$
with $f_n(z)= f(z,v)^{\b_n(z)}$, and $\b_n(z)\in\mathrm{PGL}_3$,
close to $\un$ for all $z$ and identically $\un$
to high order whenever $z\in (f_n)^{-1}(R)$.
As such the $f(z,v_n)$ lift to $S$, and
lifting is a Zariski closed condition,
so we have a lifting $F(z,v):\D\times V\ra S$ of $f(z,v)$.
Of course,
no $f_n$ need equal any $F(z,v_n)$, but the
distance between their graphs goes to zero
at points of bubbling, so,
recalling $S'=S$
for ease of notation, it will suffice
to answer the question for some 
meromorphic mapping, 
$$
F:X:=\D\times\D : (z,t) \mpo f_t(z)\,\, (\in S\bsh\{B\cup\mathrm{sing}(S)\},
\,\,\, \text{if,}\,\, t\neq 0)
$$
from a bi-disc, defined everywhere except the origin,
with  bubbles  forming as $t\ra 0$. Put
$t$ to be the 2nd projection, and consider
a resolution $\phi$ of $F$,
$$
\begin{CD}
\D @<<t< X @<<\pi< \tilde{X} @ <<\tilde{\pi}< \cX\\
@. @. @VV{\phi}V @VV{\tilde{\phi}}V \\
@. @. S@<<< \cS
\end{CD}
$$
Where $\tilde{\phi}$, and $\tilde{\pi}$ are defined
by normalising the dominant component of $\tilde{X}\times_S \cS$.
As such, if we suppose, as  we may, that $\pi$ is a
minimal resolution of $F$ obtained by blowing up
in closed points, then $\cX\ra \tilde{X}$ has
pure ramification, and this only over components, $c$,
of $(t\pi)^{-1}(0)$ which are contracted to 
points in the singularities of $S$. Consequently,
at the price of replacing $\cX$ by an almost \'etale
bi-rational cover, we may, \cite[XIII,5.3]{sga1}, suppose that $\cX$
is smooth, and:
$$K_{\cX} = \tilde{\pi}^* K_{\tilde{X}} + \sum_c (1-\frac{1}{n_c}) c$$
for some $n_c\in\bn$. Similarly $(\cS,\cB)$ has simple
normal crossings so for $\{b\}$ the components of
$\phi^{-1}(B)$, and, implicitally, omitting $c$
 if it already appears among the $b$'s,
$$
K_{\tilde{X}} + \sum_b b + \sum_c (1-\frac{1}{n_c}) c =
\phi^*(K_S+B) + F + G
$$
for $F$ and $G$ (possibly nil) effective $\bq$-divisors with the
former supported in the fibre over $t\pi$, and the
latter transverse to it. Among
the $b$'s and $c$'s there may or may not be the
proper transform $\tilde{\D}$ of the central fibre, $\D$,
of $t$, but were it amongst the $b$'s we
put $a=1$, should it be among the $c$'s (so,
implicitly not a $b$) we put $a=1-1/n_c$,
for the appropriate $c$, and 0 otherwise. In all
cases,
$$K_{\tilde{X}}+ a\, \tilde\D\, =\, \pi^* (K_X+a\,\D) + D_0$$
for $D_0$ an effective divisor contracted by $\pi$,
so that for,
$$D=D_0 + \sum_{b\neq\tilde{\D}} b + \sum_{c\neq \tilde{\D}} (1-\frac{1}{n_c}) c$$
and $e$ any divisor contracted by $\pi$,
$K_S+B$ nef. yields:
$$(D-F)\cdot e\geq 0$$
so, as ever, the negativity of the intersection
form forces $F=D+E$ for $E$ effective, possibly nil, and
contracted by $\pi$. Now suppose there is 
a bubble in the boundary, then it must be
one of the above $b$'s, and $F$ can have
no support on it, contrary to the definition
of $D$.\end{proof}

As such we may refine \ref{fact:nobloch} for surfaces
by way of,

\begin{fact}\label{fact:surbloch} {\em Let $(S,B)$,
with log-canonical singularities, be
a model of a quasi-projective surface 
with $K_S+B$ nef., and $(\tilde{S},\tb)$ the modification
in the elliptic Gorenstein singularities, $P$, if any,
which is isomorphic to the minimal model around
a neighbourhood of the same, and $(S,B)$ otherwise,
then should the
Bloch principle fails, {\it i.e.} GG, \ref{summaryd:crit}.(d),
holds but \ref{summaryd:crit}.(b) or (c) fails (the 
former, rather than the latter, will be the case of
interest iff $K_S+B$ is not strictly positive
on the boundary) then there is a closed positive current 
$T$ on a smooth champ $\tilde{\cS}\ra \tilde{S}$ 
over $\tilde{S}$ such that all
of \ref{fact:nobloch}.(a)-(b) hold with $W$ of
op. cit. exactly the boundary components, if any,
on which $K_S+B$ is nil. In particular if $P\neq\emptyset$,
any components of $T$ in the fibres of $\tilde{S}\ra S$ are
a chimera, {\it i.e.}, 
by construction, they can be blown down.} 
\end{fact}

\subsection{Proof of the principle}\label{SS:bloch}

Let $X$ be an algebraic surface, and $T$ a closed
positive current in the Nevanlinna sense arising
from a sequence of discs $f_n:\D\ra X$ at some
radius $r$ as in \ref{prop:closed},
and suppose, in addition, that there is a simple normal
crossing divisor $D\sbs S$ such that $T=\un_D T$,
or, equivalently, $T=\sum_i \lb_i D_i$, for
$D_i$ the components of $D$, and $\lb_i\geq 0$.
As in \ref{fact:taut2}, we let $f'_n:\D\ra P:=\bp(\Om^1_X(\log D))$
be their logarithmic derivatives, with $L$ the tautological
bundle, and since our goal is a lower bound for
the (Nevanlinna) degree of $(f'_n)^*\mathrm{c}_1(L)$
in terms of the (Nevanlinna) area on $X$, we may,
without loss of generality, suppose:
\begin{equation}\label{eq:limsup}
\limsup_n \left(\nint_{\D (r)} f_n^* \om_X\right)^{-1} \nint_{\D (r)} (f'_n)^*
\mathrm{c}_1(\ov{L}) < \infty
\end{equation}
Consequently as per \ref{sch:systems} we get a
well defined derivative $T'$ on $P$ lifting
$T$, and we assert,
\begin{prop}\label{prop:main}
If 
$(X,D)$ is obtained from a pair $(Y,B)$ of
a simple normal crossing divisor $B$ on a smooth
surface $Y$ by blowing up in the crossings
of $B$, and 
the origins $f_n(0)$ are bounded away from $D$, then:
$$
L\cdot T'\geq (K_X + D)\cdot T
$$
\end{prop}
\begin{proof}
For every curve $C$ in the support of $D$ there is
a short exact sequence,
$$
0 \ra \cO_C (K_X + D) \ra \Om_X (\log D) \otimes \cO_C \ra \cO_C \ra 0
$$
where the last map is the residue map, which in
turn defines a section $S_C$ in the fibre $P_C$
of $P$ over $C$. The $S_C$ are disjoint among
themselves, and:
$$
\cO_{P_C} (S_C) = \cO_{P_C} (L - \pi^* (K_X + D))
$$
The map $\rho:X\ra Y$ distinguishes curves in $D$
into those arising as proper transforms of curves
in $B$, and those blown down by $\rho$. For $C$
of the latter type, $S_C$, and whence $L$, are
nef. on $P_C$. Now define,
$$
T_C:= \begin{cases}
\un_{P_C} T & \text{$C$ not contracted by $\rho$,}\\
\un_{P_C} T - \sum_{C'\in \vert D\vert\bsh C} \un_{\pi^{-1}(C\cap C')} T
& \text{otherwise}
\end{cases}
$$
then $T=\sum_C T_C$, and identifying
a curve $B_i$ in $B$ with its proper transform,
$$
L\cdot T\geq \sum_i L\cdot T_{B_i} \geq (K_X+D)\cdot T + \sum_i S_{B_i}\cdot T_{B_i}
$$
where, at a risk of notational
confusion with \ref{fact:cart}, 
 here and elsewhere $\un_* T$ is the part
of $T$ supported on the set $*$, which itself
is the push-forward of a closed positive 
current on $*$ if $*$ is, inter alia, Zariski
closed.

We thus require to prove that the $S_{B_i}\cdot T_{B_i}\geq 0$,
to which end let's momentarily forget all of
the above and make,
\begin{inter}\label{inter:1}
{\em
Let $V\sbs W$ be a proper subvariety, 
and $F\sbs V$ a Cartier divisor on $V$.
For each $n\in\bn$ let $\pi_n: W_n\ra W$
be the blow up in $\cO_V(-nD)$ viewed
as a sheaf of ideals on $W$, with $E_n$
the exceptional divisor. If, for simplicity
of exposition, we suppose all of
$F,V,W$ 
smooth with $V$ a divisor,
then $W_n$ has only a $\bz/n$ quotient
singularity on $E_n$, and this is disjoint
from the proper transform $\tilde{V}$ of
$V$. As such there is a smallest smooth
champ $\cW_n\ra W_n$ on which currents
{\it etc.} may be understood without
repeated use of the preparation theorem.

Given where we started, we obviously 
suppose there is a current $T$ arising
from discs on $W$ by way of \ref{prop:closed},
with $T_n$ liftings to $\cW_n$ as in
\ref{sch:systems}, with origins bounded
away from $F$, and we assert:
\begin{claim}\label{claim:nonneg}
If for every $n\in\bn$, $(\pi_n)_* (\un_{E_n\bsh\tilde{V}}T_n)=0$
then, $F\cdot\un_V T\geq 0$.
\end{claim}
\begin{proof}
For each $n$ write:
$$
T_n= \un_{\tilde{V}} T_n + \un_{E_n\bsh\tilde{V}}T_n + \un_{\cW_n\bsh\pi^{-1}(V)} T
$$ 
then by \ref{fact:cart}, $E_n\cdot T_n\geq 0$,
while $E_n$ is negative on the fibres of $\pi_n$
so,
$$
E_n\cdot \un_{\tilde{V}} T_n + E_n \un_{\cW_n\bsh\pi^{-1}(V)} T\geq 0
$$
and the first term is $n F\cdot \un_V T$, while,
$
\pi_n^* V = \tilde{V} + n E_n,
$
so the second is bounded above by, $\frac{1}{n} V\cdot \un_{W\bsh V} T$.
\end{proof}
}
\end{inter}
The intermission over, we see that we'll be
done if we can establish the conditions 
of \ref{claim:nonneg} in the particular
case of $W=P$, $V=P_C$, and $F=S$ for
$C$ in the support of our boundary. To this end suppose locally 
$D$ is given by an equation $x=0$, and $y$ is the other coordinate. In 
a neighbourhood of $S$ we have a coordinate $t = \frac{dy}{dx/x}$ (or 
$\frac{dy/y}{dx/x}$ if we're close to a crossing of $D$) with respect 
to which $\pi_{n} : P_n \ra P$ is a resolution of 
the ideal $I_n = (t^n , x)$, i.e. $\cO_{P_n}(-E_n)=\pi^{-1}(I_n)$, so to be an open 
neighbourhood a distance $\e$ off the proper
transform of $P_C$ amounts to the condition,
$$
\frac{\vert x \vert}{\vert x \vert + \vert t \vert^n} \geq \e 
\Leftrightarrow \vert dy \vert \ (\hbox{resp.} \ \left\vert 
\frac{dy}{y} \right\vert \ \hbox{near a crossing}) \leq \left( 
\frac{1-\e}{\e} \right)^{\frac{1}{n}} \ \frac{\vert dx \vert}{\,\,\,\vert 
x \vert^{1-1/n}} \, .
$$
Consequently we need an estimate for 
the measure, $T \left( \un_{\d , \e} \, \frac{dd^c\vert x\vert ^2}{\vert x 
\vert^{2-2/n}} \right)$ for $\un_{\d , \e}$ a small neighbourhood a 
distance $\d$ around $C$, but $\e$ off 
the proper transform of $P_C$. 
This is however rather easy, {\it e.g.}
by \ref{eq:complete2}, the complete metric is
absolutely integrable so the above satisfies:
$
\ll \d^{2/n}\log^2\frac{1}{\d}
$.
\end{proof}
From which we can proceed to our main application:
\begin{cor}\label{cor:bloch}
Let $(S,B)$,
with log-canonical singularities, be
a model of a quasi-projective surface 
with $K_S+B$ nef.,
and suppose 
there is a proper sub-variety $Z\sbs S$ 
without generic points in the boundary such
that the Green-Griffiths property, \ref{summaryd:crit}.(d),
holds for $(S,B)$ then $S\bsh B$ is hyperbolic
modulo Z, \ref{summaryd:crit}.(b), and it is
complete hyperbolic modulo $Z$, \ref{summaryd:crit}.(c),
iff $(K_S+B)\cdot C >0$ for every curve $C\sbs B$.
\end{cor}
\begin{proof} We aim for  \ref{summaryd:crit}.(b) first,
and suppose it fails.
Let, $\rho:X\ra S$ be a smooth model such
that the total transforms of $B$, any
elliptic Gorenstain singularities, $P$, and $Z$, form a
smooth simple normal crossing divisor $D$, 
and $\rho$ factors through some $X\ra Y$ as
hypothesised in \ref{prop:main}. 
We divide $D$ into the part $\pa$ supported
on $\pi^{-1}(B\cup P)$, and the rest $F$.
By the definition of log-canonical singularities,
there is a divisor $E\geq 0$ contracted by $\rho$
such that,
\begin{equation}\label{eq:discrep}
(K_X+\pa)= \rho^*(K_S+B) + E
\end{equation}
By \ref{fact:nobloch} we can find a closed positive current $T$
on $X$ supported in $D$
afforded by discs $f_n$ at a radius $r$ in the Nevanlinna sense, \ref{eq:cnev},
such that the origins $f_n(0)$ are bounded away from $D$,
and by \ref{lem:Dorigin} we may even suppose that the
$f'_n(0)$ are uniformly bounded below. Consequently,
\ref{fact:taut2} and  \ref{fact:taut3} together with
\ref{claim:lang1} apply to find a derived current $T'$ on
$\bp(\Om^1_X(\log D))$  lifting $T$ such that,
$$
L\cdot T'\leq \limsup_n \left( \nint_{\D (r)} f_n^* \om_X\right)^{-1} \mathrm{Rad}_{f_n}^F(r)
$$
Since the origins are bounded away from $D$, so, a fortiori
$F$, by \ref{fact:cart}, the right hand side is at most
$F\cdot T$, so that by \ref{prop:main}, and \ref{eq:discrep},
$$
(K_S+B)\cdot \rho_* T\leq (K_X+D)\cdot T - F\cdot T\leq L\cdot T' - F\cdot T\leq 0
$$
Since GG holds for $(S,B)$ it must have log-general type
and admits a canonical model $\nu:(S,B)\ra (S'',B'')$,
so that the support of $T$ is contracted by $\nu$,
and $T^2<0$. However, all curves contracted by $\nu$ necessarily
belong to $Z\cup B$, so by virtue of the positions
of the $f_n(0)$, $T^2\geq 0$ by \ref{fact:cart}, which
is nonsense. 

As such, let us turn to the discussion of complete
hyperbolicity. The only if has been covered in \ref{summaryd:crit},
and by \ref{prop:nobubbles} together with an appeal to
\cite{duval} as in
\ref{fact:nobloch} the current $T$ of \ref{fact:surbloch}
cannot be supported in the boundary $\tb$ 
of the modification, $\tilde{S}\ra S$, in \ref{fact:surbloch}
except when this differs from $S$, {\it i.e.}
in the presence of elliptic Gorenstein
singularities. Let us clarify this latter, and
very minor, difficulty: supposing these exist
in $S$, then $(\tilde{S},\tilde{B})\ra (S,B)$
is a modification
in the elliptic Gorenstein singularities, $P$, alone,
$K_{\tilde{S}}+\tb$ is the pull-back of $K_S+B$., the pre-image
of every point $p\in P$ is a connected component,
$\tb_p$ of $\tb$, and $\tilde{S}$ has at worst quotient
singularities. In particular, we have a Vistoli
covering champ $\tilde{\cS}\ra \tilde{S}$, and if $\tz$ is the
proper transform of the locus $Z\sbs S$ and we
arrive to prove that discs are arbitrarily close
to $\tz+\tb_P$ in $\tilde{S}$, then, plainly they're arbitrarily
close to $Z$. As such, the problem of elliptic
Gorenstein singularities may be ignored by
replacing $(S,B)$ by $(\tilde{S},\tilde{B})$, and every occurrence
of close to $Z$ by close to $Z$ and elliptic
Gorenstein components of the boundary. 
A more serious difficulty is that the
above upper bound for the intersection with $L$
may fail because of \ref{eq:complete3}.
To minimise the
changes, let's proceed by induction, {\it viz:},
bearing in mind the above remarks about the elliptic
Gorenstein case,  
prove arbitrary proximity to $Z\cup B'$ for 
$B'\sbs B$ a (eventually empty) divisor on $S$.
We can confine our attention to discs close
to a connected component of $Z\cup B'$, and,
as in \ref{factd:vistoli} we take an almost \'etale map
$\cS\ra S$ from a champ de Deligne-Mumford
in which the boundary $\cB$ becomes simple
normal crossing. Now fix a component $C\sbs B'$,
with $\cC\ra C$ the champ over it, and put $B=C+C'$, $B'=C+B''$. Away from
$C$ we blow up as before with a view to arguing
as in \ref{prop:main}. Around $C$ we proceed
quite differently. Firstly $\cC$ may have
nodes. These are disjoint from the rest
of $B$ but may intersect $Z$, if we modify
in them (blow up at the $\cS$ level and
take moduli) we get some exceptional divisor
$N$, together with proper transforms $\tilde{C}$, $\tilde{Z}$,
of $C$, respectively $Z$,
and if we can prove that large discs must
be arbitrarily close to
$\tilde{Z} + N + B''$, then on blowing down
the discs are either small because they're
close to a node, or they're arbitrarily
close to $Z+B''$. As such we may, without
loss of generality, suppose
that $\cC$ is smooth, and for $n\in\bn$,
we let $p_n: \cS_n\ra \cS$ be the champ
obtained by taking a $n$th root along $\cC$.
Of course $Z$ may be rather far from simple
normal crossing in $\cS_n$, so we blow
up the champ $\cS_n$ to some $\rho:\cX\ra\cS$ in which
the total transform $\cD$ of $Z+C'$ is
simple normal crossing, {\it i.e.} we're
planning on running the same argument as
before but with $\cS_n$ instead of $S$,
and $C'$ instead of $B$. The singularities
remain log-canonical so this time for
$\pa$ the part of $\cD$ supported in
$\rho^{-1}(C')$, with $F$ again denoting
the rest, there is a divisor $E$ contracted
by $\rho$ such that,
$$
(K_\cX + \pa) =\rho^* (K_{\cS_n} + C') +E$$
By \ref{sch:systems}.(c), we again have
a closed current $T_n$ on $\cS_n$ arising
from big discs, which we may suppose to
have origins bounded away from $Z+C'$, and
whence, $V\cdot T_n\geq 0$ for every divisor
$V$ supported in $Z+C'$.
Next we take a logarithmic derivative,
$T_n'$ on $\pi:\bp(\Om^1_{\cX}(\log \cD))\ra \cX$ with
$L$ the tautological bundle, and, although
the origins of the discs can be close to $C$, we
still have the bound
$L\cdot T_n'\leq F\cdot T_n$ 
because close to $C$ \ref{eq:length} applies
rather than \ref{eq:complete3}, so,
if we write,
$$
T_n'=\un_{\pi^{-1}(\cD)} T_n' + R_n
$$
then $L\cdot \un_{\pi^{-1}(\cD)} T_n'\geq (K_\cX +\cD)\cdot \un_\cD T_n$
for the same reason as \ref{prop:main}, while
by \ref{fact:surbloch}, 
(the possibly nil) current $R_n$ is supported
in at most fibres of $\pi$. As such,
$$
0\geq \left(K_{S} + (1-\frac{1}{n})C + C'\right)\cdot T
$$
and the images of the $T_n$ in the homology
of $S$ are constant, so
taking $n\ra\infty$ we get exactly the
same contradiction as before.
\end{proof}

\subsection{Kobayashi metric}\label{SS:last}

Since we'll have a certain need for a good
definition of tangent bundle, let's start
with $(\cX,\cD)$ a smooth champ de Deligne-Mumford
with simple normal crossing boundary, and
recall,

\begin{defn}\label{defn:kob}
Let $(\cX,\cD)$ be as above, then for $x \in 
\cX \bsh \cD $, $t \in T_
{\cX} (-\log \cD)\otimes \bc (x)$,
$$
{\mathrm{Kob}}(t)_{\cX\bsh\cD,x} = \inf \left\{ 
\frac{1}{\vert R \vert}\, \vert\, f : 
(\D , 0) \ra (\cX \bsh \cD,x) : f_* \left( 
\frac{\partial}{ \partial z} 
\right) = R t \right\} \, .
$$
\end{defn}

A useful tool for the calculation of which is,

\begin{fact}\label{fact:kobcalc1}{\em
Let $x\in\cX\bsh\cD$ be as above, and suppose
that every sequence of pointed discs in $\cX\bsh\cD$
with origin $x$ admits a subsequence converging in $\cX$
uniformly on compact sets of $\D$, then for  any (\'etale)
neighbourhood $V\ra \cX\bsh\cD$ of $x$, there is a constant
$c(x,V)>0$ such that,
$$
c(x,V)\, \mathrm{Kob}_{V,x}\leq \mathrm{Kob}_{\cX\bsh\cD} \leq \mathrm{Kob}_{V,x}$$ 
}
\end{fact}
\begin{proof} The inequality on the right is automatic,
{\it i.e.} by definition the Kobayashi metric increases
under any mapping to $\cX\bsh\cD$. As to the left, 
if it fails, we may suppose
that there are a sequences of discs, $f_n:(\D,0)\ra (\cX\bsh\cD,x)$,
$g_n: (\D,0) \ra V$, with:
$$
(f_n)_* (\dz)= R_n (g_n)_* (\dz) ,\,\, R_n\ra \infty$$
with the $g_n$'s as close to an extremal disc in
the given direction as we please. By hypothesis,
however, after subsequencing, $f_n$ converge
to a disc. Whence there is 
some $0<r<1$, such that for all sufficiently large
$n$, $f_n(\D (r))\sbs V$,
so $R_n$ is at most $(1+\e)/r$, for $\e>0$ only
depending on how far the $g_n$ are from being extremal.
\end{proof}

Let us also observe that this continues to
work for $x\in \cD$, to wit:
\begin{fact}\label{fact:kobcalc2} {\em Everything as
in \ref{fact:kobcalc1}, except 
we suppose the aforesaid convergence for pointed
discs in $\cX\bsh\cD$ with origin in some relative compact
$U\sbs \cX\bsh\cD$ to a disc in $\cX$, then for any (\'etale)
neighbourhood $V\ra \cX$ of the closure of $U$, there is a constant
$c(U,V)>0$ such that for all $u\in U$,
$$
c(U,V)\, \mathrm{Kob}_{V\bsh\cD,u}\leq \mathrm{Kob}_{\cX\bsh\cD,u} 
\leq \mathrm{Kob}_{V\bsh\cD,u}$$
}
\end{fact}
\begin{proof} As above, but take the origins in $U$. \end{proof}

and, in fact, this really has nothing to do with
$\cD$ being a divisor, {\it i.e.}

\begin{fact}\label{fact:kobcalc3} {\em The only place
in the above where we're using that $\cD$ is a 
(simple normal crossing) divisor or $\cX$ is smooth
is in order for $T_\cX (-\log\cD)$ to be a bundle, and
whence the infinitesimal form, \ref{defn:kob},
of the Kobayashi semi-distance is well defined.
The integrated form of the distance is, however,
always well defined, so, the above makes perfect
sense even if $\cX$ is singular, or $\cD$ is
any sub champ, in the integrated form, or,
even infinitesimal form if the only condition
is $\cX\bsh\cD$ smooth and irrespective of the
co-dimension of $\cD$.} 
\end{fact}

There are, however, differences in the behaviour
of the Kobayashi metric according to the 
co-dimension of the boundary, but in the
presence of completeness such as \ref{defn:bloch}.(c),
or \ref{summaryd:crit}.(c), what the
difference may be is by \ref{fact:kobcalc1}-\ref{fact:kobcalc3}
a wholly local problem. As such let us discuss
some such problems by way of,

\begin{rmk}\label{rmk:koblocal} {\em Suppose
the boundary were indeed a simple normal
crossing divisor, then, by the uniformisation
theorem, for $V$ as above sufficiently small
$\Kob_{V\bsh \cD}$ is complete and comparable
above and below to \ref{eq:complete1}.
Thus,

{\bf (a)} Hypothesis 
as in \ref{fact:kobcalc2} and say $O\sbs U$ open, 
$\Kob_{\cX\bsh\cD}$ is complete in $O$. 

\noindent At the other extreme, if $\cD$ were a
smooth point, $p$, say, to avoid 
notational confusion, of $\cX$, then completeness
is impossible since no matter how
small $V$, $\Kob_{V\bsh p} =\Kob_V$,
and we have a further confirmation of
the need to fill the boundary in
\ref{summaryd:crit}. On the other
hand if it were an elliptic Gorenstein
singularity, then for $V$ small,
$\Kob_{V\bsh p}$ is complete. Indeed,
up to the action of a finite group,
the cases are two: the minimal smooth
resolution of the singularity has an
exceptional divisor, $E$, an elliptic
curve, or a cycle of rational curves,
which arise as cusps
on quotients of the
ball, respectively the bi-disc, so:

{\bf (b)} Under the hypothesis
of \ref{fact:kobcalc3} and say $\cD$
a point $p$, then if $\cX$ is
smooth at $p$, $\Kob_{\cX\bsh p}$ is
never complete in the complement of
any open by $p$. If, however, $p$
is elliptic Gorenstein, then 
$\Kob_{\cX\bsh p}$ would be
complete in a neighbourhood of $p$,
and for $\rho:\tcx\ra \cX$ a minimal
resolution of the singularity, $\rho^*\Kob_{\cX\bsh p}$
is compatible above and below with
\ref{eq:complete1}.

\noindent Finally, consider a case where there
is a difference between \ref{defn:bloch},
and \ref{summaryd:crit}, {\it e.g.} $X$
of the former is a surface with a
Duval singularity, $p$, obtained by
contracting some divisor $F$ on the
minimal smooth resolution, $Y$, and $X$
is the minimal model of $Y\bsh F$. As such, in \ref{summaryd:crit},
we should look to the Vistoli cover, \ref{factd:vistoli}, $\cX\ra X$,
around which $p$ is smooth, and
everything is as in (b) above. In
\ref{defn:bloch}, however, we're 
really working with the Kobayashi
distance in $X$, which may not have
an infinitesimal form, and is locally,
whence also globally under hypothesis
such as \ref{fact:kobcalc3}, different
from the Kobayashi distance in $\cX$.}
\end{rmk}

The upshot of which is that the definitions
\ref{defn:bloch}, and \ref{summaryd:crit}
are best possible, the mechanism employed in,
for example, 
\cite{green}, to create
``counterexamples'' to the 
Bloch principle is clear, while \ref{summaryd:crit}
implies that
Kobayashi's intrinsic metric on 
the canonical model $(S, B)$ of a quasi-projective surface
with elliptic Gorenstein locus $P$ is complete on
$\cS \bsh\{B\cup P\}$,
and
compatible above and below to the K\"ahler-Einstein
metric around $B \cup P$ provided 
that GG holds, and
we're bounded away from the
exceptional set $Z$. Whence,
it only remains to investigate the degeneration
around $Z$. To this end, modulo eschewing
any manifestly unnecessary blow ups off $Z$, let $\rho:X\ra S$ be
as in the proof of \ref{cor:bloch}, with,
$(X,D){\build\ra_{}^{p}} (S_0,B_0){\build\ra_{}^{q}} (S,B)$
the factorisation of $\rho$ through a minimal
modification of any elliptic Gorenstein singularities,
{\it  i.e.} modulo change of notation,
$(\tilde{S},\tb)$ of \ref{fact:surbloch}.
Further let,
$D$, $\pa$, $\zeta$ be the total transforms of $B_0+q^*Z$, 
$B_0$, $Z$, respectively, in $X$, and, again in $X$,  
$\tb$,   
the proper transform of $B_0$, with $\check{Z}$
defined by $D=\pa+\check{Z}$. Unravelling all of \ref{fact:taut1}-\ref{fact:taut3},
according to a fashion suggested by
\ref{prop:main} amounts for $f:\D\ra X\bsh\pa$
and $f':\D\ra\bp (\Om^1_X (\log D))$ to the inequality,
\begin{equation}\label{eq:degen1}
\log [\frac{\Vert f_*(\dz)\Vert\un_{\check{Z}}}{\vert\log\un_{\zeta}\vert}](0)
+ \nint_{\D (r)} (f')^*\mathrm{c}_1 (\ov L)- f^*\mathrm{c}_1 (\check{Z})
\leq 
\log \frac{l^c(r)}{r} \cdot\nint_{\D(r)} f^*\om
\end{equation}
where the implied metricisation, $\Vert\,\Vert$
of $\Om^1_X (\log D)$ is smooth around
components of $\z$, complete around components
in $\tb$,  the distance
functions to divisors are appropriately bounded
above as post \ref{eq:complete2}, and $l^c$ is the
length of the boundary of $\D (r)$ in a complete
metric for $X\bsh D$. 
By \ref{fact:kobcalc1}-\ref{fact:kobcalc3},
and \ref{eq:misha} we know exactly what 
the Kobayashi metric looks like unless
there are maps $f_n$ with unbounded Nevanlinna
area for any $r$, and origins $f_n(0)$ 
close
to $\z$.
In this scenario an easy case to 
distinguish- since it gives a better
bound almost trivially- is when 
presented with such a sequence
of discs $f_n$ with large derivatives
at the origin,
for some $r$ outside a
set of finite hyperbolic measure, one finds
the condition,
\begin{equation}\label{eq:degen2}
\limsup_n \left( \nint_{\D(r)} f_n^*\om_{X\bsh\tb}\right)^{-1} 
\nint_{\D (r)} f_n^*\mathrm{c}_1 (\ov L) =\infty
\end{equation}
where the metric $\om_{X\bsh\tb}$ is
complete on $X\bsh\tb$. Should this occur,
then, apart from the term at the
origin, 
the tautological degree dominates
everything in \ref{eq:degen1} except possibly some,
$(1+\d)\log\log\log \un_{\z}$ at the
origin, for any $\d>0$, arising from
the error in \ref{eq:complete2}, and \ref{claim:lang1}, so that
in this case, there must be a constant
$C(\d)>0$, such that for $n$ sufficiently large,
we have the bound:
\begin{equation}\label{eq:degen3}
\left[\frac{\Vert f_{n*}(\dz)\Vert\un_{\check{Z}}}
{\vert\log\un_{\z}\vert (\log\vert\log\un_{\z}\vert)^{1+\d}    }\right](0) 
\leq C(\d)
\end{equation}
As such we may otherwise suppose for
any sequence of discs that we have
to deal with (normalised not quite as in
\ref{eq:cint}, but rather on dividing
by the Nevanlinna area of $\om_{X\bsh\tb}$) 
that as per \ref{sch:systems},
we'll be able to lift any limiting
current $T$ of discs to a closed
positive derivative $T'$, or
more precisely, since $T$ must
be supported on $\z$, a current
away from $\z\cap \tb$, while
at such singularities the background
metric on fibres of the projective
tangent space $P:=\bp(\Om^1_X(\log D))$
may be degenerate, {\it i.e.} $\mathrm{c}_1(\ov L)$
in the above metricisation, so around
such points we only suppose that $T$
is finite valued on forms which can
be bounded above and below by a linear
combination of $\mathrm{c}_1(\ov L)$,
and $\om_{X\bsh\tb}$.
Thus, we're varying
the proof of 
\ref{prop:main} in order
to get a 
bound on the degeneracy of
the Kobayashi metric which is
very close to \ref{eq:degen3}.
To this end
for $C\sbs \z\sbs D$ as towards the end
of the proof of \ref{prop:main}, and $t\geq 0$ the
metricisation of the residual
section $S_C\sbs P_C$ defined
by the choice of norm on $\Om^1_X(\log D)$ 
consider
the function,
\begin{equation}\label{eq:degen4}
\phi_C:=\max\{-\log\vert\log\un_C\vert,\, \log t\}
\end{equation}
where in an abuse of notation we 
identify $t$ with a lifting to $P$,
since
all liftings result in the same
behaviour close to $S_C$, {\it i.e.}
comparable psh. functions achieving
the value $-\infty$ only on $S_C$, and,
supposing suitable normalisation
be it of the supremum whether of
$\un_C$ or $t$ an additional lower
bound,
\begin{equation}\label{eq:degen5}
\ddc\phi_C + N\om_P \geq 0
\end{equation}
on the whole of $P$ for some $N\in \bn$,
where, $\om_P=\mathrm{c}_1(\ov L) + N\om_{X\bsh \tb}\geq 0$, 
for the same
appropriately large $N$.
Now, for much the same reasoning
as that which leads from \ref{eq:degen2}
to \ref{eq:degen3}, we may without
loss of generality suppose that
for $r$ outside a set of finite
hyperbolic measure,
\begin{equation}\label{eq:degen6}
\liminf_n \left(\nint_{\D (r)} f_n^*\om_{X\bsh\tb}\right)^{-1} f_n^*\phi_C(0)> -\infty
\end{equation}
since otherwise, we can subsequence
to obtain for any $\e>0$ and $n$
sufficiently large, the lower bound,
\begin{equation}\label{eq:degen7}
\nint_{\D (r)} (f'_n)^*\mathrm{c}_1 (\ov L) \, \geq \,
\nint_{\D (r)} (f_n)^*\om_{X\bsh\tb}\, +\, \e (f_n^*\phi_C)(0)
\end{equation}
which, up to an adjustment for
$\exp(\e\phi)$ in the origin, leads to a similar bound 
to \ref{eq:degen3} for the
derivative at the origin,
the exact form of which we eschew, since,
like \ref{eq:degen3}, it's better than
our final bound, \ref{prop:degen}, which will correspond to 
the case $\e=1$. As such,  
supposing
\ref{eq:degen6}, the non-negative
absolutely continuous 1-form defined
by the left hand side of
\ref{eq:degen5} has, for every $n$,
a Nevanlinna integral of the same
order as that of $\om_{X\bsh\tb}$, independently
be it of $n$ or $r$. In particular,
even though there is no algebraic
variety on which the $\mathrm{c}_1(E_n)$
of \ref{inter:1} would be equal to
$-\ddc\phi_C$ Lebesgue  almost everywhere
(equivalently we would like to replace $\vert x\vert^{1/n}$
in op. cit. by $\vert\log\vert x\vert\vert^{-1}$) the proof
of \ref{claim:nonneg} remains valid 
provided that,

\noindent{\bf (a)} The Nevanlinna area in
the complete metric on $X\bsh D$ remains
comparable to that of $\om_{X\bsh\tb}$,
whence guaranteeing analogous conditions 
to \ref{claim:nonneg}. 

\noindent{\bf (b)}  We change appropriately the adjunction formula for
the chern class of a residual section
according to the above metricisation of $L$.

\noindent The first  point we have
already encountered how to proceed, {\it i.e.} should this fail then by 
\ref{eq:complete3} we could replace
$\phi_C$ in \ref{eq:degen7} by $\log\vert\log\un_\z\vert$,
and obtain a better estimate on
the degeneration of the Kobayashi metric
than that which will ultimately be
proposed. As to the second item,
the problem is only relevant at
the finite set $\z\cap\tb$, at which
we can model $\om_{X\bsh\tb}$ on
a complete K\"ahler-Einstein metric
of curvature $-1$, with a K\"ahler
coordinate along the component of
$\z$, and from which a metric as
prescribed on $\Om^1_{X}(\log D)$,
satisfying adjunction whereby,
the metricisation of $\cO(K_X+D)$ 
as a sub-bundle of $\Om^1_{X}(\log D)\vert_C$
is, up to a bounded continuous
function, the restriction of the
determinant.
Consequently, we may proceed
as per \ref{prop:main} to achieve,
\begin{equation}\label{eq:degen8}
\nint_{\D (r)}  (f'_n)^*\mathrm{c}_1 (\ov L) -
(f_n)^*\mathrm{c}_1 (\ov{K_X+D}) \geq   (f_n^*\phi)(0)
+ \circ (\nint_{\D(r)} (f_n)^*\om_{X\bsh\tb})
\end{equation}
Writing $D=\pa + \check{Z}$, and taking $\ov{E}$ as in
\ref{eq:discrep}, albeit smoothly metricised,
we have for 
$(X,D){\build\ra_{}^{p}} (S_0,B_0){\build\ra_{}^{q}} (S,B)$,
the above factorisation of $\rho$, and
$K_{S_0}+B_0=q^*(K_{S}+B)$ 
metricised either in the K\"ahler-Einstein metric
or a model thereof,
\begin{equation}\label{eq:degen9}
\mathrm{c}_1\left(\ov{K_X+\pa}-p^*\ov{K_{S_0}+B_0}-\ov{E}\right) =
\ddc(\log\log^2\un_{B_0} -\log\log^2\un_{\tb} + \psi)
\end{equation}
where $\psi$ is bounded above and below, so,
say, below by zero, whence:
$$
\nint_{\D (r)} (f_n)^*\mathrm{c}_1 (\ov{K_X+\pa}) \,\geq
\nint_{\D (r)} (f_n)^*\mathrm{c}_1 (\ov{K_{S}+B}) 
- \log f_n^*\left[\frac{\vert\log\un_{B_0}\vert}{\un_E\vert\log\un_{\tb}\vert}
\right](0)
$$
The
intersections along $\check{Z}$ in \ref{eq:degen8}
and \ref{eq:degen1} cancel, we
concede $\phi\geq -\log\vert\log\un_\z\vert$
to avoid a bound that might better  that
afforded by the supposition implying \ref{eq:degen8},
use the ampleness of $K_{S}+B$ as in
the proof of \ref{cor:bloch},
and so as per
\ref{eq:degen3}, 
\begin{prop}\label{prop:degen}
Suppose the infinito property GG of \ref{summaryd:crit}.(d) holds for a
quasi-projective algebraic surface, then on the canonical
model $(S,B)$, of the necessarily general type pair, 
the Kobayashi-metric is bounded above
by a constant times the K\"ahler-Einstein metric, and
conversely, below, except around the (minimal) exceptional locus $Z$, where,
notations as above,
for every $\d>0$ 
there is a constant $c(\d)>0$ such that at worst: 
$$
\rho^*\Kob_{S\bsh\{B\cup P\}}\geq 
c(\d) 
\left[\frac{\un_{\check{Z}}\un_E\vert\log\un_{\tb}\vert}
{\vert\log\un_{\z}\vert^2 (\log\vert\log\un_{\z}\vert)^{1+\d}\vert\log p^*\un_{B_0}\vert}\right]
\tau 
$$
For $\tau$ a smooth metric on $\Om_X(\log D)$ except
around the components $\tb$ of the proper transform
of $B_0$, {\it i.e.} of $B$ and the elliptic
Gorenstein singularities, $P$, where it is mildly degenerate 
being modelled 
on a complete metric, \ref{eq:complete1}, around the
said components.
\end{prop}

\end{document}